\documentclass[11pt,reqno]{amsart}
\usepackage[margin=1.1in]{geometry}
\usepackage{amsmath,amssymb,amsthm,mathtools}
\usepackage[activate={true,nocompatibility},final,tracking=false,kerning=false,spacing=false,factor=1100,stretch=10,shrink=10,expansion=false]{microtype}
\usepackage{enumitem}
\usepackage{xcolor}
\usepackage{tikz}
\usetikzlibrary{cd}
\usepackage[colorlinks=true,linkcolor=blue!60!black,citecolor=blue!60!black,urlcolor=blue!60!black]{hyperref}

\theoremstyle{plain}
\newtheorem{theorem}{Theorem}[section]
\newtheorem*{theoremstar}{Theorem}
\newtheorem{proposition}[theorem]{Proposition}
\newtheorem{lemma}[theorem]{Lemma}
\newtheorem{corollary}[theorem]{Corollary}
\theoremstyle{definition}
\newtheorem{definition}[theorem]{Definition}
\newtheorem{example}[theorem]{Example}
\newtheorem{remark}[theorem]{Remark}

\newcommand{\g}{\mathfrak{g}}
\newcommand{\britume}{\mathfrak{t}}
\newcommand{\gdual}{\mathfrak{g}^{*}}
\newcommand{\tdual}{\mathfrak{t}^{*}}
\newcommand{\R}{\mathbb{R}}

\newcommand{\C}{\mathbb{C}}
\newcommand{\dF}{d_{F}}
\newcommand{\Om}{\Omega}
\newcommand{\Ham}{\mathrm{Ham}}
\newcommand{\Lie}{L_{\infty}}
\newcommand{\vperp}{\varpi}
\newcommand{\io}{\iota}
\newcommand{\Lder}{\mathcal{L}}
\newcommand{\vs}{\varsigma}

\newcommand{\ip}[2]{\langle #1,\,#2\rangle}
\newcommand{\pair}[2]{\left\langle #1,\,#2\right\rangle}
\newcommand{\dG}{d_{G}}
\newcommand{\dGF}{d_{G}^{F}}
\newcommand{\ka}{\kappa}

\numberwithin{equation}{section}

\begin{document}

\title[Reduction of relative multisymplectic manifolds]{Reduction of relative multisymplectic manifolds}

%% ---------- PLACEHOLDER: author & affiliation ----------
\author{Dinamo Djounvouna}
\address{Department of Mathematics, University of Manitoba, Winnipeg, Canada}
\email{djounvod@myumanitoba.ca}
%% --------------------------------------------------------


\begin{abstract}
We extend the multisymplectic reduction theory of Blacker -- itself the extension
of Marsden--Weinstein--Meyer reduction to $k$-plectic manifolds -- to the setting
of \emph{relative} multisymplectic geometry, in which a smooth map $F\colon M\to N$
carries a closed nondegenerate relative $(k{+}1)$-form
$\vperp=(\omega,\eta)$ in the mapping-cone complex of $F$. We introduce the Leibniz
algebra of relative Hamiltonian pairs and the associated relative moment maps
$\mu\in\Om^{k-1}(F,\gdual)$, and prove a relative multisymplectic reduction
theorem: for a closed equivariant level $\phi\in\Om^{k-1}(F,\gdual)$, the level
pair of $\mu$ descends to a reduced smooth map $F_\phi\colon M_\phi\to N_\phi$
carrying a unique reduced closed relative form $\vperp_\phi$; remarkably, the
horizontality of the trivializing component is forced by the relative closedness
of the level. We further prove relative analogues of the reduction of dynamics,
of the structure theory of split moment maps $\mu=\nu\cdot\ka$ -- for which the
splitting datum is a one-step cocycle in the relative Cartan model, so that split
relative Hamiltonian $G$-spaces carry canonical relative homotopy moment maps --
and of the Duistermaat--Heckman-type variation formula: with respect to suitable
conjugate distributions, the reduced relative class satisfies
$\partial_\lambda[\vperp_\psi]=\ip{c}{\lambda}\cdot[\ka_\psi]$, where $c$ is the
Chern form of the target model bundle acting through the natural
$\Om(N)$-module structure of the mapping cone. Finally, we transport the exact
stationary phase approximation to the target component and prove a genuinely
relative localization constraint: the fixed-point contributions of the pulled-back
split package on the source cancel identically.
\end{abstract}

\maketitle
\setcounter{tocdepth}{2}
\tableofcontents

%% ============================================================
\section{Introduction}\label{sec:intro}
%% ============================================================

The Marsden--Weinstein--Meyer theorem \cite{MW,Meyer} removes symmetry from a
symplectic manifold with Hamiltonian $G$-action along a level set of the moment
map. Blacker \cite{Blacker} extended this mechanism to multisymplectic
($k$-plectic) manifolds: a moment map is now a $\gdual$-valued $(k{-}1)$-form,
levels are prescribed closed forms $\phi\in\Om^{k-1}(M,\gdual)$, and the reduced
space carries a canonical closed -- though possibly degenerate -- $(k{+}1)$-form.
The same paper develops the dependence of the reduced space on the level
(a Duistermaat--Heckman-type variation formula \cite{DH}) and equivariant
localization results for the tractable class of \emph{split} moment maps
$\mu=\nu\wedge\eta$.

In parallel, the author has developed \emph{relative} multisymplectic geometry
\cite{DjThesis,DjJGP}: the geometry of a smooth map $F\colon M\to N$ equipped with
a closed nondegenerate relative $(k{+}1)$-form
\[
  \vperp=(\omega,\eta)\ \in\ \Om^{k+1}(F):=\Om^{k+1}(N)\oplus\Om^{k}(M),
  \qquad
  \dF(\omega,\eta)=(d\omega,\,F^*\omega-d\eta)=0,
\]
together with the accompanying Lie $k$-algebras of relative observables, relative
homotopy moment maps \cite{RelHMM}, and a spectrum of applications
\cite{RelApps}. The guiding example is the quasi-Hamiltonian geometry of
group-valued moment maps \cite{AMM}, where $\vperp=(\eta_G,\omega)$ is built from
the Cartan $3$-form on the target group.

The purpose of the present article is to carry out, in full, the program of
\cite{Blacker} in the relative setting: \emph{reduction of relative
multisymplectic manifolds}. Both the source and the target are reduced
simultaneously, and the reduced object is again a smooth map with a closed
relative form. Three structural phenomena, invisible in the absolute theory,
organize the results:
\begin{enumerate}[label=(\alph*),leftmargin=2.2em]
\item the trivializing component of the reduced relative form exists
\emph{because} the level is relatively closed: the $M$-side horizontality in the
reduction theorem is exactly the second component of $\dF\phi=0$;
\item the mapping cone $\Om^\bullet(F)$ is a differential graded module over
$(\Om^\bullet(N),d)$, and all split-moment-map theory -- including the
Duistermaat--Heckman variation and the localization package -- is governed by
this module structure;
\item split relative moment maps are one-step cocycles in the relative Cartan
model of \cite{RelHMM}, so that split relative Hamiltonian $G$-spaces carry
canonical relative \emph{homotopy} moment maps refining the Leibniz-algebraic
comoments considered here.
\end{enumerate}

\subsection{Main results}
Let $G$ be a compact Lie group acting on $F$ (i.e.\ on $M$ and $N$ with $F$
equivariant) and preserving a relative pre-$k$-plectic form $\vperp$, with
relative moment map $\mu\in\Om^{k-1}(F,\gdual)$
(Definition~\ref{def:relmoment}). Our main theorem is:

\begin{theoremstar}[Relative multisymplectic reduction; Theorem~\ref{thm:reduction}]
Let $\phi\in\Om^{k-1}(F,\gdual)$ be $\dF$-closed and $G$-equivariant, and let
\[
  L_N=\{y\in N:\mu_N(y)=\phi_N(y)\},
  \qquad
  L_M=F^{-1}(L_N)\cap\{x\in M:\mu_M(x)=\phi_M(x)\}
\]
be the level pair. If $L_N\subseteq N$ and $L_M\subseteq M$ are embedded
submanifolds on which $G$ acts freely, then $F$ descends to a smooth map
$F_\phi\colon M_\phi\to N_\phi$ between the quotients, and there is a unique
relative $(k{+}1)$-form $\vperp_\phi=(\omega_\phi,\eta_\phi)$ on $F_\phi$ with
\[
  i_N^*\omega=\pi_N^*\omega_\phi,
  \qquad
  i_M^*\eta=\pi_M^*\eta_\phi;
\]
moreover $\vperp_\phi$ is $d_{F_\phi}$-closed.
\end{theoremstar}

As in the absolute case the reduced form may be degenerate, and the hypotheses can
be weakened considerably (Remark~\ref{rem:weaken}). We also prove the relative
reduction of dynamics (Theorem~\ref{thm:dynamics}): invariant Hamiltonian
$F$-pairs tangent to the level pair descend, together with their Hamilton
equations.

For split relative moment maps $\mu=\nu\cdot\ka$ -- where
$\nu\in C^\infty(N,\gdual)$, $\ka\in\Om^{k-1}(F)$ is $\dF$-closed, and $\cdot$
denotes the module action of Lemma~\ref{lem:module} -- we prove: the structure
theory of levels (Proposition~\ref{prop:splitlevels}); the fixed-point theorem
for torus actions (Proposition~\ref{prop:fixedpoints}); and the following
refinement of \cite[Prop.~3.10]{Blacker}, which connects the present Leibniz-type
theory with the $L_\infty$-theory of \cite{RelHMM}:

\begin{theoremstar}[Theorem~\ref{thm:onesteptie}]
If the splitting $\mu=\nu\cdot\ka$ is basic \textup{(}$\ka$ invariant and
horizontal\textup{)}, then $\vperp-\mu$ is a one-step cocycle in the relative
Cartan model; consequently a split relative Hamiltonian $G$-space carries a
canonical relative homotopy moment map with components
$f_j=\vs(j)\,\io(v_{\xi_1}\wedge\dots\wedge v_{\xi_{j-1}})\mu_{\xi_j}$.
\end{theoremstar}

Section~\ref{sec:variation} develops the variation of the reduced space. After
establishing the relative variation formula (Lemma~\ref{lem:variation}), we
introduce conjugate distributions adapted to the relative setting and prove:

\begin{theoremstar}[Variation of the relative reduced space;
Theorem~\ref{thm:DH}]
For a torus $T$, a relative $k$-plectic Hamiltonian $T$-space with invariant
level $\phi$, a $T$-invariant $\dF$-closed $\ka\in\Om^{k-1}(F)$, and reduction
parameters $\psi\in C\cdot\ka+\phi$, if the level pairs trivialize as a family
modeled on that of $\phi$ and the fundamental distribution on $M$ is strongly
conjugate to a distribution $\underline{\britume}^*$ with respect to $\vperp$ and
$\ka$, then
\[
  \partial_\lambda[\vperp_\psi]\;=\;\ip{c}{\lambda}\cdot[\ka_\psi],
  \qquad\lambda\in C,
\]
in $H^{k+1}(\Om(F_\psi))$, where $c\in\Om^2(N_\phi,\britume)$ is the Chern form
of the target model bundle $L_N\to N_\phi$ and $\cdot$ is the module action in
cohomology. The source-side Chern data is the pullback $F_\phi^*c$, so that no
independent invariant appears on $M$.
\end{theoremstar}

For $M=\emptyset$ this recovers \cite[Thm.~5.6]{Blacker}, and for $k=1$,
$M=\emptyset$, the Duistermaat--Heckman theorem \cite{DH} in its cohomological
form. Finally, Section~\ref{sec:localization-body} transports the localization package:
the split datum exponentiates to an equivariantly closed element
$e^{z(s-\nu)}\cdot\ka$ of the relative Cartan model (Lemma~\ref{lem:eqclosed};
note the sign, dictated by our conventions), the exact stationary phase
approximation holds on the target (Theorem~\ref{thm:stationaryphase}), and a
genuinely relative constraint appears on the source: the fixed-point
contributions of the pulled-back split package cancel identically
(Theorem~\ref{thm:cancellation}).

\subsection{Relation to the literature}
This article is the fourth in a series \cite{DjJGP,RelHMM,RelApps} and is
patterned on \cite{Blacker}, whose numbering we track throughout to ease
comparison; the absolute theory is recovered verbatim at $M=\emptyset$
(and, one categorical step down, at $M=\{*\}$). Multisymplectic manifolds and
their Hamiltonian structures are treated in \cite{CIL,RyvkinWurzbacher,Rogers};
alternative moment map notions are compared in \cite{CFRZ,MS,RW,RelHMM}. Our
conventions follow \cite{CFRZ,RelHMM}; the dictionary to the conventions of
\cite{Blacker} is given in Remark~\ref{rem:dictionary}.

For the reader who has not read the earlier articles in detail, the following
table records the definitions and results imported from each companion paper;
all other statements are proved here or are standard.
\begin{center}
\renewcommand{\arraystretch}{1.25}
\begin{tabular}{@{}p{0.24\textwidth}p{0.68\textwidth}@{}}
\hline
\textbf{Source} & \textbf{Imported into this paper}\\
\hline
\cite{DjJGP} (foundations) &
relative Cartan calculus and module structure
(\S\ref{subsec:calculus}--\ref{subsec:module}); the Leibniz algebra of
relative Hamiltonian forms (Lemma~\ref{lem:leibniz}); the strong-nondegeneracy
condition used in the observable-reduction theorem.\\
\cite{RelHMM} (moment maps) &
relative homotopy moment maps and their component/Cartan-model equations, used
in the split-moment-map correspondence (Theorem~\ref{thm:onesteptie}) and the
comparison with the $L_\infty$ theory (Remark~\ref{rem:observables}).\\
\cite{RelApps} (applications) &
relative cycles, the relative Stokes theorem, and the relative integration
pairing (\S3 there), used in the variation and localization sections
(Corollary~\ref{cor:periods}, \S\ref{sec:localization-body}).\\
\cite{Blacker} (absolute reduction) &
the pattern of the reduction theorem and the numbering convention; recovered
verbatim at $M=\varnothing$.\\
\hline
\end{tabular}
\end{center}

\subsection*{Acknowledgements}

The author is deeply grateful to Derek Krepski for introducing the author to relative multisymplectic geometry during their doctoral studies and for many stimulating discussions, insightful suggestions, and continuous encouragement throughout the development of this work and the broader research program on relative multisymplectic geometry. His guidance has been invaluable and has profoundly influenced the direction of this research.

%% ============================================================
\section{Background: relative multisymplectic geometry}
\label{sec:background}
%% ============================================================

We recall the framework of \cite{DjJGP,RelHMM} and establish one new structural
ingredient, the module calculus of Lemma~\ref{lem:module}. All manifolds are
$C^\infty$ and all Lie groups compact, so that all actions are proper.

\subsection{The relative Cartan calculus}\label{subsec:calculus}
For a smooth map $F\colon M\to N$, the mapping-cone complex is
\begin{equation}\label{eq:cone}
  \Om^j(F):=\Om^j(N)\oplus\Om^{j-1}(M),
  \qquad
  \dF(\alpha,\beta):=(d\alpha,\;F^*\alpha-d\beta) .
\end{equation}
An $F$-pair of vector fields $v=(v_N,v_M)$ consists of $F$-related fields; the
relative contraction and Lie derivative
\begin{equation}\label{eq:ops}
  \io_v(\alpha,\beta)=(\io_{v_N}\alpha,\,-\io_{v_M}\beta),
  \qquad
  \Lder_v(\alpha,\beta)=(\Lder_{v_N}\alpha,\,\Lder_{v_M}\beta)
\end{equation}
satisfy all the Cartan identities, in particular
$\Lder_v=\dF\io_v+\io_v\dF$ and $[\Lder_u,\io_v]=\io_{[u,v]}$
\cite[Prop.~3.5]{RelHMM}. We freely use the vanishing principle
\cite[Rem.~3.2]{RelHMM}: a $\dF$-closed relative $0$-form is a locally constant
function on $N$ pulling back to zero on $M$, hence vanishes when $N$ is connected
and $M\neq\emptyset$.

\begin{definition}\label{def:relkplectic}
A $\dF$-closed $\vperp=(\omega,\eta)\in\Om^{k+1}(F)$ is a \emph{relative
pre-$k$-plectic structure}; it is \emph{$k$-plectic} if for every $x\in M$ the
map $w\mapsto(\io_{T_xF(w)}\,\omega,\ \io_w\,\eta)$ is injective on $T_xM$. A
relative $(k{-}1)$-form $\sigma$ is \emph{Hamiltonian}, with Hamiltonian $F$-pair
$v_\sigma$, if
\begin{equation}\label{eq:hamiltonian}
  \dF\sigma=-\io_{v_\sigma}\vperp .
\end{equation}
We write $\Ham^{k-1}(F,\vperp)$ for the space of Hamiltonian relative
$(k{-}1)$-forms. Componentwise, \eqref{eq:hamiltonian} reads
$d\sigma_N=-\io_{v_N}\omega$ and $d\sigma_M=F^*\sigma_N+\io_{v_M}\eta$.
\end{definition}

If $\vperp$ is nondegenerate then $v_\sigma$ is unique, and
$\Lder_{v_\sigma}\vperp=0$ by the magic formula. Nondegeneracy will play no role
in the reduction theorem itself (cf.\ \cite[Rem.~4.2]{Blacker}), but is used in
the structure theory of split moment maps.

\begin{remark}[Dictionary with the absolute theory]\label{rem:dictionary}
Our conventions are those of \cite{CFRZ,RelHMM}: fundamental vector fields
$v_\xi|_x=\frac{d}{dt}\big|_0\exp(-t\xi)\cdot x$, so that $\xi\mapsto v_\xi$ is a
Lie algebra homomorphism -- the same convention as \cite{Blacker} -- and
Hamiltonian forms obey \eqref{eq:hamiltonian}. In \cite{Blacker} a Hamiltonian
form satisfies $d\alpha=\io_{X_\alpha}\omega$; thus for $M=\emptyset$ a
Hamiltonian pair in our sense corresponds to a Hamiltonian form of
\cite{Blacker} with $X_\alpha=-v_\alpha$, and our (co)moment maps correspond to
those of \cite{Blacker} under $\mu\mapsto-\mu$. All statements below reduce at
$M=\emptyset$ to their counterparts in \cite{Blacker} through this dictionary,
and we will not repeat the translation each time.
\end{remark}

\begin{remark}[Caveat lector: slot order]\label{rem:slotorder}
As in the prequantization paper of this series, relative forms are
written \emph{target component first},
$(\alpha,\beta)\in\Om^j(N)\oplus\Om^{j-1}(M)$, with the
differential~\eqref{eq:cone}; this is the mirror of the
source-first convention of \cite{DjJGP,RelHMM}, the two being
exchanged by the flip composed with the sign involution
$\beta\mapsto(-1)^{|\beta|}\beta$.  We adopt the target-first order
because levels, quotients, and localization data are anchored on the
target, the source carrying the trivializing tier.
\end{remark}

\begin{remark}[On signs]\label{rem:signs}
The constructions of this paper involve several interacting grading and
contraction conventions, and the five sign-sensitive families of formulas ---
the module action \eqref{eq:module}, the conjugate-distribution identities of
\S\ref{subsec:conjugate}, the variation formula \eqref{eq:DH}, the equivariant
exponential of Lemma~\ref{lem:eqclosed}, and the source-cancellation formula of
Theorem~\ref{thm:cancellation} --- have been verified against the two boundary
reductions that fix all signs unambiguously: $M=\varnothing$, on which every
formula must return its counterpart in \cite{Blacker,BGV} under the dictionary
of Remark~\ref{rem:dictionary}, and $N=\mathrm{pt}$, on which it must return the
source theory of \cite{DjJGP}. Each formula reduces correctly on both edges; we
flag the relevant edge check at the point of use.
\end{remark}

\subsection{The module calculus of the mapping cone}\label{subsec:module}

The following structure, though elementary, is one of the organizing mechanisms
of the paper: the differential graded module action of the target de~Rham
algebra on the mapping cone controls the split moment maps
(\S\ref{subsec:split}), the variation formula (\S\ref{sec:variation}), the
localization package (\S\ref{sec:localization-body}), and, throughout, the way
target data act on source data. We state it once, in full, and refer back to it
at each use. It is used constantly in
Sections~\ref{sec:hamiltonian}--\ref{sec:localization-body}.

\begin{lemma}[Module calculus]\label{lem:module}
The assignment
\begin{equation}\label{eq:module}
  \gamma\cdot(\alpha,\beta)\;:=\;
  \bigl(\gamma\wedge\alpha,\ (-1)^{|\gamma|}\,F^*\gamma\wedge\beta\bigr),
  \qquad \gamma\in\Om^{|\gamma|}(N),
\end{equation}
makes $\Om^\bullet(F)$ a differential graded module over
$(\Om^\bullet(N),\wedge,d)$:
\begin{enumerate}[label=\textup{(\roman*)},leftmargin=2.2em]
\item $\dF(\gamma\cdot c)=d\gamma\cdot c+(-1)^{|\gamma|}\,\gamma\cdot\dF c$;
\item $\io_v(\gamma\cdot c)=(\io_{v_N}\gamma)\cdot c
      +(-1)^{|\gamma|}\,\gamma\cdot\io_v c$ for every $F$-pair $v$;
\item $\Lder_v(\gamma\cdot c)=(\Lder_{v_N}\gamma)\cdot c
      +\gamma\cdot\Lder_v c$;
\item $(\gamma\wedge\gamma')\cdot c=\gamma\cdot(\gamma'\cdot c)$ and
      $1\cdot c=c$.
\end{enumerate}
\end{lemma}

\begin{proof}
(iv) is immediate from \eqref{eq:module} and
$F^*(\gamma\wedge\gamma')=F^*\gamma\wedge F^*\gamma'$. For (i), write
$c=(\alpha,\beta)$ and expand both sides using \eqref{eq:cone}; the
$N$-components agree by the Leibniz rule on $N$, while the $M$-component of the
left side is
$F^*\gamma\wedge F^*\alpha-(-1)^{|\gamma|}\bigl(F^*d\gamma\wedge\beta
+(-1)^{|\gamma|}F^*\gamma\wedge d\beta\bigr)
=F^*\gamma\wedge(F^*\alpha-d\beta)-(-1)^{|\gamma|}F^*d\gamma\wedge\beta$,
which is the $M$-component of the right side. For (ii), the $N$-components agree
by the Leibniz rule for $\io_{v_N}$; the $M$-component of the left side is
$-\io_{v_M}\bigl((-1)^{|\gamma|}F^*\gamma\wedge\beta\bigr)
=-(-1)^{|\gamma|}F^*(\io_{v_N}\gamma)\wedge\beta
-F^*\gamma\wedge\io_{v_M}\beta$,
using $\io_{v_M}F^*=F^*\io_{v_N}$ for $F$-related fields; this matches the
$M$-component of the right side, namely
$(-1)^{|\gamma|-1}F^*(\io_{v_N}\gamma)\wedge\beta
+(-1)^{|\gamma|}F^*\gamma\wedge(-\io_{v_M}\beta)$.
Finally (iii) follows from (i)--(ii) and the magic formula, or directly
componentwise using $\Lder_{v_M}F^*=F^*\Lder_{v_N}$.
\end{proof}

\begin{remark}\label{rem:functions}
For $f\in C^\infty(N)=\Om^0(N)$ the action is
$f\cdot(\alpha,\beta)=(f\alpha,\,(F^*f)\beta)$, with
$\dF(f\cdot c)=df\cdot c+f\cdot\dF c$. All $\gdual$-valued constructions below
use \eqref{eq:module} extended coefficientwise.
\end{remark}

%% ============================================================
\subsection{Notation}\label{subsec:notation}
For the reader's convenience, the recurring notation of the paper:
\begin{center}
\renewcommand{\arraystretch}{1.2}
\begin{tabular}{@{}llll@{}}
\hline
$F\colon M\to N$ & the underlying map &
$\vperp=(\omega,\eta)$ & relative form, target first\\
$F_L\colon L_M\to L_N$ & $F$ restricted to the level pair &
$\mu=(\mu_N,\mu_M)$ & relative moment map\\
$F_\phi\colon M_\phi\to N_\phi$ & the reduced map &
$\phi=(\phi_N,\phi_M)$ & the level\\
$i=(i_N,i_M)$ & inclusions of the level pair &
$v_\xi=(v_\xi^N,v_\xi^M)$ & fundamental $F$-pair\\
$\pi=(\pi_N,\pi_M)$ & quotient projections &
$\mu=\nu\cdot\ka$ & split moment map\\
$\vperp_\phi=(\omega_\phi,\eta_\phi)$ & reduced relative form &
$c$ & Chern form of the level bundle\\
\hline
\end{tabular}
\end{center}
Componentwise, every construction pairs a target object (no decoration
conflicts) with its source companion, and the map between them is always the
relevant restriction of $F$.

%% ============================================================
\section{Relative Hamiltonian $G$-spaces}\label{sec:hamiltonian}
%% ============================================================

This section assembles the algebraic cast of the paper in three steps, each
building on the last: first the \emph{Hamiltonian pairs} and their Leibniz
bracket (the observables); then the \emph{relative moment maps} (equivariant
families of Hamiltonian data); then the \emph{split} moment maps (those
factoring through the module calculus), which are the ones that later support
variation and localization. The reader may keep this three-step picture in
mind throughout.

\subsection{The Leibniz algebra of relative Hamiltonian forms}

Following \cite[\S2]{Blacker} we equip $\Ham^{k-1}(F,\vperp)$ with the bracket
\begin{equation}\label{eq:bracket}
  \{\sigma,\tau\}\;:=\;\Lder_{v_\sigma}\tau ,
\end{equation}
where $v_\sigma$ is a (the, in the nondegenerate case) Hamiltonian $F$-pair of
$\sigma$.

\begin{lemma}\label{lem:leibniz}
Let $(F,\vperp)$ be relative pre-$k$-plectic and suppose Hamiltonian $F$-pairs
have been chosen linearly \textup{(}no choice is needed when $\vperp$ is
nondegenerate\textup{)}. Then:
\begin{enumerate}[label=\textup{(\roman*)},leftmargin=2.2em]
\item $\{\sigma,\tau\}$ is again Hamiltonian, with Hamiltonian $F$-pair
$[v_\sigma,v_\tau]$;
\item $\{\,,\}$ satisfies the \textup{(}left\textup{)} Leibniz identity
$\{\sigma,\{\tau,\rho\}\}=\{\{\sigma,\tau\},\rho\}+\{\tau,\{\sigma,\rho\}\}$;
\item the failure of antisymmetry is exact:
$\{\sigma,\tau\}+\{\tau,\sigma\}
=\dF\bigl(\io_{v_\sigma}\tau+\io_{v_\tau}\sigma\bigr)$.
\end{enumerate}
Thus $(\Ham^{k-1}(F,\vperp),\{\,,\})$ is a Leibniz algebra and
$\sigma\mapsto v_\sigma$ intertwines $\{\,,\}$ with the bracket of $F$-pairs.
\end{lemma}

\begin{proof}
(i) Using $[\dF,\Lder_{v_\sigma}]=0$, the Hamilton equation for $\tau$,
$[\Lder_u,\io_v]=\io_{[u,v]}$, and $\Lder_{v_\sigma}\vperp=0$:
\[
  \dF\{\sigma,\tau\}
  =\Lder_{v_\sigma}\dF\tau
  =-\Lder_{v_\sigma}\io_{v_\tau}\vperp
  =-\io_{[v_\sigma,v_\tau]}\vperp-\io_{v_\tau}\Lder_{v_\sigma}\vperp
  =-\io_{[v_\sigma,v_\tau]}\vperp .
\]
(ii) is the operator identity
$[\Lder_{v_\sigma},\Lder_{v_\tau}]=\Lder_{[v_\sigma,v_\tau]}
=\Lder_{v_{\{\sigma,\tau\}}}$ applied to $\rho$, using (i). (iii) By the magic
formula and the Hamilton equations,
$\Lder_{v_\sigma}\tau+\Lder_{v_\tau}\sigma
=\dF(\io_{v_\sigma}\tau+\io_{v_\tau}\sigma)
+\io_{v_\sigma}\dF\tau+\io_{v_\tau}\dF\sigma$,
and the last two terms are
$-\io_{v_\sigma}\io_{v_\tau}\vperp-\io_{v_\tau}\io_{v_\sigma}\vperp=0$
by anticommutativity of relative contractions.
\end{proof}

Having the observables, we now let a group act: a relative moment map is an
equivariant family of Hamiltonian forms, one for each generator of the action,
and the whole reduction theory rests on the interplay between its two
components.

\subsection{Relative moment maps}

\begin{definition}\label{def:relmoment}
Let $G$ act on $F$ (equivariantly on $M\to N$) preserving the relative
pre-$k$-plectic form $\vperp$. A \emph{relative comoment map} is a homomorphism
of Leibniz algebras
$\widetilde\mu\colon\g\to\Ham^{k-1}(F,\vperp)$
lifting the fundamental assignment, i.e.\ with $v_{\widetilde\mu(\xi)}=v_\xi$ for
all $\xi\in\g$. Packaging $\widetilde\mu$ as
$\mu\in\Om^{k-1}(F,\gdual):=\bigl(\Om^{k-1}(N)\oplus\Om^{k-2}(M)\bigr)\otimes
\gdual$
via $\mu_\xi:=\widetilde\mu(\xi)$, we call $\mu$ a \emph{relative moment map}
when it is additionally $G$-equivariant (for the coadjoint action on $\gdual$
tensored with the natural actions on relative forms), and we call
$(F,\vperp,G,\mu)$ a \emph{relative \textup{(}multisymplectic\textup{)}
Hamiltonian $G$-space}. Explicitly, the defining relations are
\begin{equation}\label{eq:momentrelations}
  \dF\mu_\xi=-\io_{v_\xi}\vperp,
  \qquad
  \{\mu_\xi,\mu_\zeta\}=\mu_{[\xi,\zeta]},
  \qquad \xi,\zeta\in\g .
\end{equation}
\end{definition}

\begin{remark}\label{rem:equivariancebracket}
For a $G$-equivariant $\mu$ the Leibniz condition in
\eqref{eq:momentrelations} is automatic: infinitesimal equivariance reads
$\Lder_{v_\xi}\mu_\zeta=\mu_{[\xi,\zeta]}$, which \emph{is}
$\{\mu_\xi,\mu_\zeta\}=\mu_{[\xi,\zeta]}$. Conversely, exactly as in
\cite[Prop.~3.4]{Blacker}, a relative Hamiltonian $\g$-space integrates to a
Hamiltonian $G$-space when $G$ is connected: the identity
$\Lder_{v_\xi}\mu_\zeta=\mu_{[\xi,\zeta]}$ integrates along one-parameter
subgroups. For $k=1$ these conditions are automatic and the comoment is unique
by the rigidity theorem of \cite[Thm.~7.1]{RelApps}.
\end{remark}

\begin{remark}[Comparison with relative homotopy moment maps]
\label{rem:comparehmm}
A relative comoment map is related to, but distinct from, the first component of
a relative homotopy moment map \cite[Def.~5.2]{RelHMM}: both satisfy the Hamilton
equation, but the Leibniz condition composes brackets through
$\Lder_{v_\xi}$, whereas the $L_\infty$-condition composes through
$\vs(2)\,\io(v_\xi\wedge v_\zeta)\vperp$ up to the higher component $f_2$. The
two conditions differ by the $\dF$-exact term
$\dF\io_{v_\xi}\mu_\zeta$, by the magic formula. In particular every
\emph{equivariant one-step extension} in the sense of
\cite[Lem.~6.4]{RelHMM} is simultaneously a relative comoment map
\textup{(}its condition \textup{(M2)} is the Leibniz identity\textup{)} and the
seed of a full relative homotopy moment map; Theorem~\ref{thm:onesteptie} below
exhibits a large class of such data.
\end{remark}

\begin{example}\label{ex:firstexamples}
\begin{enumerate}[label=\textup{(\roman*)},leftmargin=2.2em]
\item \textup{(Absolute case.)} For $M=\emptyset$,
Definition~\ref{def:relmoment} reduces to the Hamiltonian $G$-spaces of
\cite[Def.~3.1]{Blacker}, via Remark~\ref{rem:dictionary}.
\item \textup{(Exact relative structures.)} If $\vperp=-\dF\theta$ for a
$G$-invariant $\theta\in\Om^{k}(F)$, then $\mu_\xi:=-\io_{v_\xi}\theta$ is a
relative moment map: the Hamilton equation is the computation of
\cite[Ex.~9.2]{RelHMM}, and equivariance follows from
$[\Lder_{v_\xi},\io_{v_\zeta}]=\io_{v_{[\xi,\zeta]}}$ and the invariance of
$\theta$. This is the relative counterpart of the canonical moment maps on
$\Lambda^k_\ell T^*E$ of \cite[Ex.~3.2(iii)]{Blacker}.
\item \textup{(Quasi-Hamiltonian $G$-spaces.)} For a quasi-Hamiltonian
$G$-space $(M,\omega,\mu_{\mathrm{AMM}})$ \cite{AMM}, the relative $2$-plectic
map $(\mu_{\mathrm{AMM}}\colon M\to G,\ \vperp=(\eta_G,\omega))$ of
\cite[\S8]{RelHMM} carries the canonical one-step datum
$\mathrm{M}(\xi)=(\tfrac12\ip{\theta^L+\theta^R}{\xi},0)$, which by
Remark~\ref{rem:comparehmm} is a relative comoment map in the present sense.
The reduction theory of Section~\ref{sec:reduction} therefore applies to
group-valued moment maps, at arbitrary \emph{form-valued} levels
$\phi\in\Om^1(\mu_{\mathrm{AMM}},\gdual)$; the classical point-level reduction
$\mu_{\mathrm{AMM}}^{-1}(e)/G$ of \cite{AMM} is not of this type, and the two
schemes complement one another.
\end{enumerate}
\end{example}
Among relative moment maps, those that factor through the module calculus as a
scalar comoment times a fixed closed form are both the most common in examples
and the ones for which the finest results (variation, localization) hold; we
isolate them now.

\subsection{Split relative moment maps}\label{subsec:split}

\begin{definition}\label{def:split}
A relative moment map $\mu\in\Om^{k-1}(F,\gdual)$ is \emph{split} if
\[
  \mu=\nu\cdot\ka
\]
for some $\nu\in C^\infty(N,\gdual)$ and some $\dF$-closed
$\ka\in\Om^{k-1}(F)$, where $\cdot$ is the module action of
Lemma~\ref{lem:module} \textup{(}extended $\gdual$-coefficientwise, cf.\
Remark~\ref{rem:functions}\textup{)}. The splitting is \emph{invariant} if
$\ka$ is $G$-invariant, and \emph{basic} if additionally
$\io_{v_\xi}\ka=0$ for all $\xi\in\g$ \textup{(}relative contraction\textup{)}.
We call the resulting datum a \emph{split relative Hamiltonian $G$-space}.
\end{definition}

Componentwise, $\mu_\xi=\nu_\xi\cdot\ka=(\nu_\xi\,\ka_N,\ (F^*\nu_\xi)\,\ka_M)$
and the Hamilton equation unpacks, via Lemma~\ref{lem:module}(i) and
$\dF\ka=0$, to
\begin{equation}\label{eq:splitmoment}
  d\nu_\xi\cdot\ka\;=\;-\,\io_{v_\xi}\vperp,
  \qquad\text{i.e.}\qquad
  d\nu_\xi\wedge\ka_N=-\io_{v_\xi^N}\omega,
  \quad
  F^*d\nu_\xi\wedge\ka_M=-\io_{v_\xi^M}\eta .
\end{equation}

\begin{proposition}\label{prop:kernel}
Let $\mu=\nu\cdot\ka$ be an invariant splitting. Then the pair of kernel
distributions
$\mathcal F=(\ker\ka_N,\ \ker\ka_M\cap\ker F^*\ka_N)$ is preserved by $G$; if
the splitting is basic and $G$ is connected, then the fundamental vector fields
satisfy $v_\xi^N\in\ker\ka_N$ and $v_\xi^M\in\ker\ka_M\cap\ker F^*\ka_N$
pointwise.
\end{proposition}

\begin{proof}
Invariance of $\ka$ gives invariance of the kernels componentwise. For a basic
splitting, the relative horizontality $\io_{v_\xi}\ka=0$ reads
$\io_{v_\xi^N}\ka_N=0$ and $\io_{v_\xi^M}\ka_M=0$; moreover
$\io_{v_\xi^M}F^*\ka_N=F^*\io_{v_\xi^N}\ka_N=0$ by $F$-relatedness.
\end{proof}

The next result upgrades \cite[Prop.~3.10]{Blacker} -- equivariant closedness of
the split package -- to a statement about the relative Cartan model
$C_G(F)$ of \cite[\S6.2]{RelHMM}, whose degree-$m$ component is
$C^m_G(N)\oplus C^{m-1}_G(M)$ with differential
$\dGF(\alpha,\beta)=(\dG\alpha,F^*\alpha-\dG\beta)$ and
$(\dG\alpha)(\xi)=d(\alpha(\xi))-\io_{v_\xi}(\alpha(\xi))$.

\begin{theorem}[Split moment maps are one-step cocycles]\label{thm:onesteptie}
Let $(F,\vperp,G,\mu)$ be a relative Hamiltonian $G$-space whose moment map
admits a basic splitting $\mu=\nu\cdot\ka$. Then
\[
  \vperp_G\;:=\;\vperp-\mu
\]
is a one-step cocycle in the relative Cartan model extending $\vperp$;
consequently, by \textup{\cite[Thm.~6.5]{RelHMM}}, the maps
\[
  f_j(\xi_1,\dots,\xi_j)
  \;=\;\vs(j)\,\io\bigl(v_{\xi_1}\wedge\dots\wedge v_{\xi_{j-1}}\bigr)\,
  \mu_{\xi_j},
  \qquad 1\le j\le k,\quad \vs(j)=-(-1)^{j(j+1)/2},
\]
constitute a canonical relative homotopy moment map
$\g\rightsquigarrow\Lie(F,\vperp)$.
\end{theorem}

\begin{proof}
By \cite[Lem.~6.4]{RelHMM} it suffices to verify, for all $\xi,\zeta\in\g$:
(M1) $\dF\mu_\xi=-\io_{v_\xi}\vperp$; (M2)
$\Lder_{v_\xi}\mu_\zeta=\mu_{[\xi,\zeta]}$; (M3) $\io_{v_\xi}\mu_\xi=0$.
Condition (M1) is the Hamilton equation and (M2) is the equivariance of $\mu$
(Remark~\ref{rem:equivariancebracket}). For (M3), Lemma~\ref{lem:module}(ii)
gives
\[
  \io_{v_\xi}\mu_\xi
  =\io_{v_\xi}\bigl(\nu_\xi\cdot\ka\bigr)
  =(\io_{v_\xi^N}\nu_\xi)\cdot\ka+\nu_\xi\cdot\io_{v_\xi}\ka
  =\nu_\xi\cdot\io_{v_\xi}\ka=0,
\]
since $\nu_\xi$ has form-degree $0$ and the splitting is basic.
\end{proof}

\begin{remark}
Theorem~\ref{thm:onesteptie} is the mechanism by which the Leibniz-algebraic
moment maps of this paper interact with the $L_\infty$-theory: for basic
splittings the two theories are generated by the same cocycle $\vperp-\mu$. In
particular all consequences of \cite[\S\S6--7]{RelHMM} -- explicit higher
components, existence and uniqueness statements, and the vanishing of the
algebraic obstruction -- apply to split relative Hamiltonian $G$-spaces.
\end{remark}

\subsection{Vanishing and fixed points}

The following relative analogue of \cite[Prop.~2.5]{Blacker} localizes zeros of
Hamiltonian pairs on the \emph{source}, where relative nondegeneracy lives. Note
that no compactness of $M$ is needed: the critical point is produced on $N$ and
transported through $F$.

\begin{proposition}\label{prop:vanishing}
Let $(F,\vperp)$ be relative $k$-plectic, let $\ka\in\Om^{k-1}(F)$ be
$\dF$-closed, let $f\in C^\infty(N)$, and suppose $\sigma:=f\cdot\ka$ is
Hamiltonian with $F$-pair $v_\sigma$. If $y\in N$ is a critical point of $f$
lying in the image of $F$, then $v_\sigma$ vanishes along the fiber over $y$:
$v_\sigma^M(x)=0$ and $v_\sigma^N(y)=0$ for every $x\in F^{-1}(y)$. In
particular, if $N$ is compact and $F$ is surjective, the vanishing set of
$v_\sigma^M$ is nonempty.
\end{proposition}

\begin{proof}
By Lemma~\ref{lem:module}(i) and $\dF\ka=0$,
$-\io_{v_\sigma}\vperp=\dF(f\cdot\ka)=df\cdot\ka$, i.e.
\[
  \io_{v_\sigma^N}\omega=-df\wedge\ka_N,
  \qquad
  \io_{v_\sigma^M}\eta=F^*df\wedge\ka_M .
\]
Let $x\in F^{-1}(y)$ with $df_y=0$. Then $\io_{v_\sigma^N}\omega$ vanishes at
$y$ and $\io_{v_\sigma^M}\eta$ vanishes at $x$ (as $(F^*df)_x=df_y\circ T_xF=0$).
Since $v_\sigma$ is an $F$-pair, $v_\sigma^N(y)=T_xF\bigl(v_\sigma^M(x)\bigr)$,
so the nondegeneracy of $\vperp$ at $x$ (Definition~\ref{def:relkplectic})
applied to $w=v_\sigma^M(x)$ forces $v_\sigma^M(x)=0$, and then
$v_\sigma^N(y)=0$. If $N$ is compact, $f$ has a critical point $y$, and
surjectivity of $F$ provides $x\in F^{-1}(y)$.
\end{proof}

\begin{proposition}[Fixed points for split torus actions]\label{prop:fixedpoints}
Let $T$ be a torus acting on the relative $k$-plectic $(F,\vperp)$ with a split
relative moment map $\mu=\nu\cdot\ka$. If $N$ is compact and $F$ is surjective,
then the fixed point set of $T$ in $M$ is nonempty.
\end{proposition}

\begin{proof}
Choose a generator $\xi\in\britume$, so that the closure of
$\exp(\R\xi)$ is $T$ and the fundamental pair $v_\xi$ generates the $T$-orbits.
The component $\mu_\xi=\nu_\xi\cdot\ka$ is of the form covered by
Proposition~\ref{prop:vanishing} with $f=\nu_\xi$; hence $v_\xi^M$ vanishes at
some $x\in M$. Since $\xi$ generates $T$ and $T$ is connected, $x$ is a fixed
point of the $T$-action on $M$.
\end{proof}

\begin{remark}
At $M=\emptyset$ (with the convention that ``$F$ surjective'' is dropped and the
vanishing is read on $N$) these statements recover
\cite[Prop.~2.5, Prop.~3.11]{Blacker}. The relative refinement is that the fixed
point is produced on the source $M$, where the nondegeneracy of $\vperp$ is
concentrated -- for instance, on a quasi-Hamiltonian $G$-space the target
$G$ carries no nondegeneracy at all, yet the argument still yields fixed points
in $M$ for split torus symmetries.
\end{remark}
%% ============================================================
\section{Relative multisymplectic reduction}\label{sec:reduction}
%% ============================================================

Throughout this section $(F,\vperp,G,\mu)$ is a relative Hamiltonian $G$-space
in the sense of Definition~\ref{def:relmoment}.

\subsection{Level pairs}

\begin{definition}\label{def:levelpair}
Let $\phi\in\Om^{k-1}(F,\gdual)$ be $\dF$-closed and $G$-equivariant (for the
same action as $\mu$). Regarding the components
$\mu_N,\phi_N$ as sections of $\Lambda^{k-1}T^*N\otimes\gdual$ and
$\mu_M,\phi_M$ as sections of $\Lambda^{k-2}T^*M\otimes\gdual$, the
\emph{level pair} of $\mu$ at $\phi$ is
\begin{equation}\label{eq:levelpair}
  L_N:=\{y\in N:\ \mu_N(y)=\phi_N(y)\},
  \qquad
  L_M:=F^{-1}(L_N)\,\cap\,\{x\in M:\ \mu_M(x)=\phi_M(x)\} .
\end{equation}
By construction $F$ restricts to an equivariant map
$F_L\colon L_M\to L_N$, and $L_N,L_M$ are $G$-invariant by the equivariance of
$\mu$ and $\phi$.
\end{definition}

We write $i_N\colon L_N\hookrightarrow N$ and $i_M\colon L_M\hookrightarrow M$
for the inclusions; the square $(i_N,i_M)\colon F_L\to F$ is an equivariant map
of pairs, so the pullback
$\Phi^*(\alpha,\beta):=(i_N^*\alpha,i_M^*\beta)$ is a cochain map intertwining
the relative Cartan calculi \cite[Prop.~10.2]{RelHMM}. Note also that a
pointwise equality of sections restricts through any map \emph{into} the
coincidence locus:
\begin{equation}\label{eq:coincidence}
  i_N^*\mu_{N,\xi}=i_N^*\phi_{N,\xi},
  \qquad
  i_M^*\mu_{M,\xi}=i_M^*\phi_{M,\xi},
  \qquad
  (F\circ i_M)^*\mu_{N,\xi}=(F\circ i_M)^*\phi_{N,\xi},
\end{equation}
the last identity because $F\circ i_M$ takes values in $L_N$.

\begin{remark}\label{rem:isotropylevels}
As in \cite[\S4]{Blacker} one may treat levels $\phi$ that are merely
$\dF$-closed, replacing $G$ by the isotropy subgroup
$G_\phi\subseteq G$ of $\phi$; every statement below adapts verbatim, at the
cost of heavier notation. We restrict to equivariant levels for clarity.
\end{remark}

\subsection{The reduction theorem}

\begin{theorem}[Relative multisymplectic reduction]\label{thm:reduction}
Let $(F,\vperp,G,\mu)$ be a relative \textup{(}pre-\textup{)}$k$-plectic
Hamiltonian $G$-space, let $\phi\in\Om^{k-1}(F,\gdual)$ be $\dF$-closed and
equivariant, and let $(L_N,L_M)$ be the level pair \eqref{eq:levelpair}. Assume
$L_N\subseteq N$ and $L_M\subseteq M$ are embedded submanifolds and that $G$
acts freely on $L_N$ and on $L_M$. Then:
\begin{enumerate}[label=\textup{(\roman*)},leftmargin=2.2em]
\item the quotients $N_\phi:=L_N/G$ and $M_\phi:=L_M/G$ are smooth manifolds and
$F_L$ descends to a smooth map $F_\phi\colon M_\phi\to N_\phi$;
\item there is a unique pair
$\vperp_\phi=(\omega_\phi,\eta_\phi)\in\Om^{k+1}(F_\phi)$ with
\begin{equation}\label{eq:reducedform}
  i_N^*\omega=\pi_N^*\omega_\phi,
  \qquad
  i_M^*\eta=\pi_M^*\eta_\phi,
\end{equation}
where $\pi_N\colon L_N\to N_\phi$ and $\pi_M\colon L_M\to M_\phi$ are the
quotient maps;
\item $\vperp_\phi$ is $d_{F_\phi}$-closed, i.e.\ a relative
pre-$k$-plectic structure on $F_\phi$.
\end{enumerate}
We call $(F_\phi,\vperp_\phi)$ the \emph{reduced relative space} and
$\vperp_\phi$ the \emph{reduced relative \textup{(}pre\textup{)}multisymplectic
structure}.
\end{theorem}

The reduction descends the whole square of level pair and quotient:
\begin{equation}\label{eq:reduction-square}
\begin{tikzcd}[row sep=2.0em, column sep=3.0em,
  /tikz/every label/.append style={font=\footnotesize}]
L_M \arrow[r, "F_L"] \arrow[d, "\pi_M"']
& L_N \arrow[d, "\pi_N"]\\
M_\phi \arrow[r, "F_\phi"']
& N_\phi
\end{tikzcd}
\qquad\qquad
\begin{tikzcd}[row sep=1.0em, column sep=1.4em,
  /tikz/every label/.append style={font=\scriptsize}]
(F,\vperp) \arrow[r]
& (L_N,L_M) \arrow[r]
& (M_\phi{\to}N_\phi) \arrow[r]
& \vperp_\phi
\end{tikzcd}
\end{equation}
the left square commuting by equivariance of $F_L$, the right chain recording
the passage from the relative Hamiltonian space through the level pair and the
quotient to the reduced relative form.

\begin{proof}
\emph{Descent of the map.}
(i) Freeness and properness ($G$ compact) give smooth quotients; $F_L$ is
equivariant, so $\pi_N\circ F_L$ is $G$-invariant and factors through a smooth
$F_\phi$ with $F_\phi\circ\pi_M=\pi_N\circ F_L$.

\emph{Descent of the form: invariance.}
(ii) We must show that $i_N^*\omega$ and $i_M^*\eta$ are basic, i.e.\
$G$-invariant and horizontal, for the respective quotients. Invariance is clear:
$G$ preserves $\vperp$, the levels are invariant, and pullback along the
inclusions commutes with the action.

\emph{Horizontality on the target.} Fix $\xi\in\g$. The fundamental field
$v_\xi^N$ is tangent to $L_N$, so contraction commutes with $i_N^*$, and the
$N$-component of the Hamilton equation
$\dF\mu_\xi=-\io_{v_\xi}\vperp$ reads $\io_{v_\xi^N}\omega=-d\mu_{N,\xi}$.
Hence, by \eqref{eq:coincidence} and the $N$-component $d\phi_{N,\xi}=0$ of
$\dF\phi_\xi=0$,
\[
  \io_{v_\xi^N}\,i_N^*\omega
  =i_N^*\bigl(\io_{v_\xi^N}\omega\bigr)
  =-\,d\,i_N^*\mu_{N,\xi}
  =-\,d\,i_N^*\phi_{N,\xi}
  =-\,i_N^*\,d\phi_{N,\xi}=0 .
\]

\emph{Horizontality on the source.} The $M$-component of the Hamilton equation
reads $\io_{v_\xi^M}\eta=F^*\mu_{N,\xi}-d\mu_{M,\xi}$ (recall
$-\io_{v_\xi}(\omega,\eta)=(-\io_{v_\xi^N}\omega,\ \io_{v_\xi^M}\eta)$ and
\eqref{eq:cone}). Since $v_\xi^M$ is tangent to $L_M$ and $F\circ i_M$ maps into
$L_N$, equations \eqref{eq:coincidence} give
\[
  \io_{v_\xi^M}\,i_M^*\eta
  =i_M^*\bigl(F^*\mu_{N,\xi}-d\mu_{M,\xi}\bigr)
  =(F\circ i_M)^*\phi_{N,\xi}-d\,i_M^*\phi_{M,\xi}
  =i_M^*\bigl(F^*\phi_{N,\xi}-d\phi_{M,\xi}\bigr)
  =0,
\]
the last equality being precisely the $M$-component of $\dF\phi_\xi=0$. Thus the
relative closedness of the level supplies the source horizontality.

\emph{Why this is a genuinely relative phenomenon.} It is worth pausing on the
mechanism, as it has no absolute analogue. In the target argument,
horizontality of $i_N^*\omega$ follows from the Hamilton equation alone, exactly
as in \cite{Blacker}: the target form is Hamiltonian-basic because $\omega$ is
the honest multisymplectic form and $\mu_N$ is its moment map. The source form
$\eta$, by contrast, is \emph{not} separately assumed to be Hamiltonian-basic,
and there is no independent source moment map forcing it to be so; its
horizontality is instead \emph{forced}, through the $M$-component
$F^*\phi_{N,\xi}=d\phi_{M,\xi}$ of the single equation $\dF\phi_\xi=0$, by the
compatibility of the source and target levels \emph{inside the mapping cone}.
In other words, the trivializing role of $\eta$ --- it exists to satisfy
$F^*\omega=d\eta$, not to carry its own symmetry data --- is exactly what makes
its descent automatic once the target descends. This is the conceptual point of
the theorem: relative closedness of the level converts target horizontality
into source horizontality at no extra cost, and the source component reduces
because it was never independent to begin with.

\emph{Geometric meaning of the reduced map.} The reduced map itself has a
transparent description: a point of $M_\phi$ is a $G$-orbit of source solutions
of the level constraint, a point of $N_\phi$ a $G$-orbit of target solutions,
and $F_\phi$ sends the orbit of $x$ to the orbit of $F(x)$ --- symmetry classes
of constrained states mapping to symmetry classes of their images. The fibres
of $F_\phi$ are the reduced fibres of $F$, and the reduced relative form
$\vperp_\phi$ is the residual coupling between them: the target component
$\omega_\phi$ is the reduced dynamics on the constrained target, and the source
component $\eta_\phi$ is the reduced trivialization tying the constrained
source to it. Reduction, in other words, preserves the entire
bulk--trivialization architecture of the relative structure, one symmetry
class at a time.

Being basic, $i_N^*\omega$ and $i_M^*\eta$ descend uniquely through the
submersions $\pi_N,\pi_M$, yielding \eqref{eq:reducedform}; uniqueness holds
because $\pi_N^*$ and $\pi_M^*$ are injective.

\emph{Closedness.}
(iii) From $d\,i_N^*\omega=i_N^*d\omega=0$ and injectivity of $\pi_N^*$ we get
$d\omega_\phi=0$. For the second component of $d_{F_\phi}\vperp_\phi$, compute
on $L_M$, using $F_\phi\circ\pi_M=\pi_N\circ F_L$ and the $M$-component
$F^*\omega=d\eta$ of $\dF\vperp=0$:
\[
  \pi_M^*\bigl(F_\phi^*\omega_\phi\bigr)
  =F_L^*\,\pi_N^*\omega_\phi
  =F_L^*\,i_N^*\omega
  =i_M^*\,F^*\omega
  =i_M^*\,d\eta
  =d\,i_M^*\eta
  =\pi_M^*\,d\eta_\phi .
\]
Injectivity of $\pi_M^*$ gives $F_\phi^*\omega_\phi=d\eta_\phi$, i.e.\
$d_{F_\phi}\vperp_\phi=0$.
\end{proof}

\begin{remark}\label{rem:degenerate}
As in the absolute theory, $\vperp_\phi$ need not be nondegenerate: relative
reduction produces relative \emph{pre}multisymplectic structures in general.
Sufficient conditions for nondegeneracy of the reduced pair constitute an
interesting open problem, mirroring \cite[\S7(1)]{Blacker}; see
Section~\ref{sec:outlook}.
\end{remark}

\begin{remark}[Weakening the hypotheses]\label{rem:weaken}
An inspection of the proof shows that only the following were used: an action
preserving $\vperp$; equivariant forms $\mu,\phi$ with
$\dF(\mu_\xi-\phi_\xi)=-\io_{v_\xi}\vperp$; and smooth structure plus freeness
along the level pair. Neither nondegeneracy nor any homogeneity of $\vperp$
enters -- the relative counterpart of \cite[Rem.~4.2]{Blacker}. In particular
the theorem applies to arbitrary closed relative forms and to actions with
merely locally free level pairs after passing to orbifold quotients.
\end{remark}

\subsection{Reduction of dynamics}

\begin{theorem}[Reduction of dynamics]\label{thm:dynamics}
In the situation of Theorem~\ref{thm:reduction}, let
$\sigma\in\Ham^{k-1}(F,\vperp)$ be $G$-invariant with $G$-invariant Hamiltonian
$F$-pair $v_\sigma=(v_\sigma^N,v_\sigma^M)$. Suppose either
\begin{enumerate}[label=\textup{(\roman*)},leftmargin=2.2em]
\item $v_\sigma^N$ is tangent to $L_N$ and $v_\sigma^M$ is tangent to $L_M$; or
\item $\{\sigma,\,\mu_\xi-\phi_\xi\}=0$ for all $\xi\in\g$, where
$\{\,,\}=\Lder_{v_\sigma}$ as in \eqref{eq:bracket}.
\end{enumerate}
Then $v_\sigma$ descends to an $F_\phi$-pair
$\bar v=(\bar v_N,\bar v_M)$ on the reduced relative space, the relative form
$\dF\sigma$ descends to $\Om^{k}(F_\phi)$, and the descended form satisfies the
reduced Hamilton equation
\[
  \overline{\dF\sigma}\;=\;-\,\io_{\bar v}\,\vperp_\phi .
\]
\end{theorem}

\begin{proof}
First note that (ii) implies (i): the condition
$\Lder_{v_\sigma}(\mu_\xi-\phi_\xi)=0$ for all $\xi$ says the flow of
$v_\sigma$ (an $F$-pair of flows intertwined by $F$) preserves the pair of
sections $\mu-\phi$ componentwise, hence preserves their coincidence loci
\eqref{eq:levelpair}; tangency follows where the loci are submanifolds.

Assume (i). Invariance and tangency let $v_\sigma|_{L}$ descend through the
free quotients to vector fields $\bar v_N,\bar v_M$, which are
$F_\phi$-related because $v_\sigma^N,v_\sigma^M$ are $F_L$-related and the
quotient squares commute; thus $\bar v$ is an $F_\phi$-pair.

Next, $\Phi^*(\dF\sigma)=-\Phi^*(\io_{v_\sigma}\vperp)$ is basic on the level
pair. Invariance is clear. For horizontality, using tangency (so that
$\io_{v_\sigma}$ commutes with $\Phi^*$), anticommutativity of relative
contractions, the Hamilton equation for $\mu_\xi$, the cochain property of
$\Phi^*$, and \eqref{eq:coincidence}:
\[
  \io_{v_\xi}\,\Phi^*\bigl(\io_{v_\sigma}\vperp\bigr)
  =-\,\Phi^*\bigl(\io_{v_\sigma}\io_{v_\xi}\vperp\bigr)
  =\Phi^*\bigl(\io_{v_\sigma}\dF\mu_\xi\bigr)
  =\io_{v_\sigma|_L}\,d_{F_L}\,\Phi^*\phi_\xi
  =\io_{v_\sigma|_L}\,\Phi^*\bigl(\dF\phi_\xi\bigr)
  =0 .
\]
Hence $\Phi^*(\dF\sigma)$ descends to a unique
$\overline{\dF\sigma}\in\Om^k(F_\phi)$ with
$\pi^*\overline{\dF\sigma}=\Phi^*(\dF\sigma)$, where
$\pi^*$ denotes the componentwise quotient pullback. Finally,
\[
  \pi^*\bigl(-\io_{\bar v}\vperp_\phi\bigr)
  =-\,\io_{v_\sigma|_L}\,\pi^*\vperp_\phi
  =-\,\io_{v_\sigma|_L}\,\Phi^*\vperp
  =\Phi^*\bigl(-\io_{v_\sigma}\vperp\bigr)
  =\Phi^*\bigl(\dF\sigma\bigr),
\]
using $\pi$-relatedness of $v_\sigma|_L$ and $\bar v$ and tangency once more;
injectivity of $\pi^*$ concludes.
\end{proof}

\begin{remark}
Condition (ii) is the relative form of the bracket criterion of
\cite[Thm.~4.3(ii)]{Blacker}. Invariant decomposable Hamiltonian
\emph{multivector} pairs can be treated by the same argument applied factorwise;
we refrain from developing the relative multivector calculus here and refer to
\cite[\S2]{Blacker} for the absolute picture.
\end{remark}

\subsection{Reduction of the observable algebra}

Theorem~\ref{thm:dynamics} reduces Hamiltonian data one observable at
a time.  We now upgrade it to a morphism of Leibniz algebras, in the
setting of Lemma~\ref{lem:leibniz}; since in the pre-$k$-plectic case
a Hamiltonian form does not determine its $F$-pair, the clean
statement is at the level of \emph{pairs}.  Write
$\widetilde\Ham{}^{k-1}(F,\vperp)
:=\{(\sigma,v):\ \dF\sigma=-\io_v\vperp\}$
with the Leibniz bracket
$\{(\sigma,v),(\tau,w)\}:=(\Lder_v\tau,[v,w])$, which satisfies the
identities of Lemma~\ref{lem:leibniz} verbatim, and let
$\Phi$, $\pi=(\pi_N,\pi_M)$ be the restriction and quotient maps of
Theorem~\ref{thm:reduction}.

\begin{definition}\label{def:reducible}
A Hamiltonian pair $(\sigma,v)$ is \emph{$\phi$-reducible} if
$\sigma$ and $v$ are $G$-invariant, $v$ is tangent to the level pair
\textup{(}condition \textup{(i)} of
Theorem~\ref{thm:dynamics}\textup{)}, and $\Phi^*\sigma$ is
horizontal, $\io_{v_\xi}\Phi^*\sigma=0$ for all $\xi\in\g$.  We write
$\widetilde\Ham{}^{k-1}_\phi(F,\vperp)$ for the set of
$\phi$-reducible pairs.
\end{definition}

\begin{theorem}[Reduction of the observable algebra]
\label{thm:observables}
Assume $G$ is connected and the hypotheses of
Theorem~\ref{thm:reduction} hold.  Then:
\begin{enumerate}[label=\textup{(\roman*)},leftmargin=2.2em]
\item $\widetilde\Ham{}^{k-1}_\phi(F,\vperp)$ is a Leibniz subalgebra
of $\widetilde\Ham{}^{k-1}(F,\vperp)$;
\item the assignment
\[
  \mathcal{R}(\sigma,v):=(\sigma_\phi,\bar v),
  \qquad
  \pi^*\sigma_\phi=\Phi^*\sigma,
  \quad
  \bar v=\text{the descent of }v|_L,
\]
is well defined, lands in
$\widetilde\Ham{}^{k-1}(F_\phi,\vperp_\phi)$, and is a morphism of
Leibniz algebras;
\item $\ker\mathcal{R}
=\{(\sigma,v)\ \phi\text{-reducible}:\ \Phi^*\sigma=0\ \text{and}\
v|_L\ \text{tangent to the orbits}\}$ is a two-sided Leibniz ideal,
and $\mathcal{R}$ induces an injective Leibniz morphism from the
quotient.
\end{enumerate}
\end{theorem}

\begin{proof}
(i)  Let $(\sigma,v),(\tau,w)$ be $\phi$-reducible.  Invariance of
$\{\sigma,\tau\}=\Lder_v\tau$ and of $[v,w]$ is clear; tangency of
$[v,w]$ holds componentwise because tangency is preserved by Lie
brackets.  For horizontality of
$\Phi^*\{\sigma,\tau\}=\Lder_{v|_L}\Phi^*\tau$ (tangency lets
$\Lder_v$ commute with $\Phi^*$): since $G$ is connected and $v$ is
invariant, $[v_\xi,v]=0$, so
\[
  \io_{v_\xi}\,\Lder_{v|_L}\,\Phi^*\tau
  =\Lder_{v|_L}\,\io_{v_\xi}\,\Phi^*\tau
  -\io_{[v,v_\xi]|_L}\,\Phi^*\tau
  =0 ,
\]
both terms vanishing by horizontality of $\Phi^*\tau$ and
$[v,v_\xi]=0$.

(ii)  $\Phi^*\sigma$ is invariant and horizontal, hence basic, and
$\sigma_\phi$ exists and is unique; $\bar v$ exists as in the proof
of Theorem~\ref{thm:dynamics}.  The pair is Hamiltonian downstairs:
using injectivity of $\pi^*$, the cochain property of $\Phi^*$ and
$\pi^*$, tangency, and the defining property of $\vperp_\phi$,
\[
  \pi^*\bigl(d_{F_\phi}\sigma_\phi+\io_{\bar v}\vperp_\phi\bigr)
  =\Phi^*\bigl(\dF\sigma\bigr)+\io_{v|_L}\,\Phi^*\vperp
  =\Phi^*\bigl(\dF\sigma+\io_{v}\vperp\bigr)=0 .
\]
For the morphism property,
$\pi^*\mathcal{R}\{\,(\sigma,v),(\tau,w)\,\}^{(1)}
=\Phi^*\Lder_v\tau=\Lder_{v|_L}\pi^*\tau_\phi
=\pi^*\Lder_{\bar v}\tau_\phi$, whence
$\mathcal{R}(\Lder_v\tau)=\Lder_{\bar v}\tau_\phi
=\{\mathcal{R}(\sigma,v),\mathcal{R}(\tau,w)\}^{(1)}$, and the
vector-field components match because descent of invariant tangent
fields intertwines brackets.

(iii)  The kernel of any morphism of Leibniz algebras is a two-sided
ideal: if $\mathcal{R}x=0$ then
\[
  \mathcal{R}\{x,y\}=\{\mathcal{R}x,\mathcal{R}y\}=\{0,\mathcal{R}y\}=0
  \quad\text{since }\Lder_{0}=0,
  \qquad
  \mathcal{R}\{y,x\}=\{\mathcal{R}y,0\}=(\Lder_{\bar v_y}0,[\bar v_y,0])=0;
\]
injectivity on the quotient holds by construction.
\end{proof}

\begin{remark}\label{rem:observables}
\textup{(a)}  When $\vperp_\phi$ is nondegenerate, pairs downstairs
are determined by their forms and Theorem~\ref{thm:observables}
becomes a Leibniz morphism
$\Ham^{k-1}_\phi(F,\vperp)\to\Ham^{k-1}(F_\phi,\vperp_\phi)$ on the
nose.  \textup{(b)}  Surjectivity onto
$\widetilde\Ham{}^{k-1}(F_\phi,\vperp_\phi)$ is an extension problem
--- every reduced pair must be realized by a reducible pair upstairs
--- which already in the absolute theory requires hypotheses; we do
not pursue it here.  \textup{(c)}  Via the hemi--semi comparison of
Lemma~\ref{lem:leibniz}\textup{(iii)} and
Theorem~\ref{thm:onesteptie}, Theorem~\ref{thm:observables} is the
Leibniz face of an $L_\infty$-level reduction of the Rogers-type
algebra of \cite{DjJGP} along the associated relative homotopy moment
map, carried out in the framework of \cite{RelApps}; see the Outlook.
\end{remark}

\subsection{Split levels}

For split moment maps the level pairs acquire the familiar preimage form.

\begin{proposition}\label{prop:splitlevels}
Let $\mu=\nu\cdot\ka$ be an invariant splitting, let $\lambda\in\gdual$ be fixed
by the coadjoint action of $G$ \textup{(}e.g.\ $G$ a torus\textup{)}, and take
the level $\phi:=\lambda\cdot\ka$, which is $\dF$-closed and equivariant. Then:
\begin{enumerate}[label=\textup{(\roman*)},leftmargin=2.2em]
\item $F^{-1}\bigl(\nu^{-1}(\lambda)\bigr)\subseteq L_M$ and
$\nu^{-1}(\lambda)\subseteq L_N$, with equality at every point where,
respectively, $\ka_M$ and $\ka_N$ are nonvanishing;
\item if $(F,\vperp)$ is relative $k$-plectic and $G$ acts locally freely along
$L_M$, then $\nu\colon N\to\gdual$ is submersive at every point of
$F(L_M)$; in particular $\nu^{-1}(\lambda)$ is smooth near $F(L_M)$;
\item under the hypotheses of Theorem~\ref{thm:reduction}, the reduced relative
space at $\phi=\lambda\cdot\ka$ is
$F_\phi\colon L_M/G\to L_N/G$ with $L_N=\nu^{-1}(\lambda)$ and
$L_M=F^{-1}(\nu^{-1}(\lambda))$ wherever $\ka$ is nonvanishing along the levels.
\end{enumerate}
\end{proposition}

\begin{proof}
(i) $\mu_N-\phi_N=(\nu-\lambda)\ka_N$ and
$\mu_M-\phi_M=(F^*\nu-\lambda)\ka_M$, whence the inclusions; conversely
$(\nu(y)-\lambda)\,\ka_N(y)=0$ forces $\nu(y)=\lambda$ when $\ka_N(y)\neq0$,
and similarly on $M$.

(ii) Let $x\in L_M$ and $\xi\in\g\setminus\{0\}$. Local freeness gives
$v_\xi^M(x)\neq0$, so by relative nondegeneracy the pair
$\bigl(\io_{v_\xi^N}\omega|_{F(x)},\ \io_{v_\xi^M}\eta|_x\bigr)$ is nonzero.
By \eqref{eq:splitmoment} this pair equals
$\bigl(-d\nu_\xi\wedge\ka_N,\ -F^*d\nu_\xi\wedge\ka_M\bigr)$ evaluated at
$(F(x),x)$; in either case the nonvanishing forces
$(d\nu_\xi)_{F(x)}\neq0$ (for the second component, because
$F^*d\nu_\xi=d\nu_\xi\circ TF$). Thus $\xi\mapsto(d\nu_\xi)_{F(x)}$ is
injective, i.e.\ $T_{F(x)}\nu$ is surjective.

(iii) Combine (i), (ii) and Theorem~\ref{thm:reduction}; nonvanishing of
$\ka_M$ along $L_M$ follows from local freeness and \eqref{eq:splitmoment}, and
of $\ka_N$ along $F(L_M)$ likewise.
\end{proof}

We close the structural development with a compatibility that upgrades the
module calculus from an upstairs tool to a feature of the reduction itself; it
is used silently in the variation and localization sections whenever reduced
split data appear.

\begin{proposition}[Reduction commutes with the module calculus]
\label{prop:module-descent}
In the situation of Theorem~\ref{thm:reduction}, let $\gamma\in\Om^\bullet(N)$
be $G$-invariant with $i_N^*\gamma$ horizontal, so that $i_N^*\gamma$ descends
to a unique $\gamma_\phi\in\Om^\bullet(N_\phi)$ with
$\pi_N^*\gamma_\phi=i_N^*\gamma$; and let $c\in\Om^\bullet(F)$ be a relative
form whose restriction $i^*c$ is basic, with descent $c_\phi\in\Om^\bullet(F_\phi)$.
Then $i^*(\gamma\cdot c)$ is basic and
\[
    (\gamma\cdot c)_\phi\;=\;\gamma_\phi\cdot c_\phi ,
\]
the module action on the right being that of the reduced map $F_\phi$.  In
particular, if $\mu=\nu\cdot\ka$ is a split moment map whose factors descend,
the reduced moment data are split, $\mu_\phi=\nu_\phi\cdot\ka_\phi$.
\end{proposition}

\begin{proof}
Restriction along the level pair is compatible with the module action: writing
$c=(\alpha,\beta)$,
\[
    i^*\bigl(\gamma\cdot c\bigr)
    =\bigl(i_N^*(\gamma\wedge\alpha),\
    (-1)^{|\gamma|}\,i_M^*(F^*\gamma\wedge\beta)\bigr)
    =\bigl(i_N^*\gamma\bigr)\cdot\bigl(i^*c\bigr),
\]
using $i_M^*F^*=F_L^*\,i_N^*$, which holds because $F\circ i_M=i_N\circ F_L$;
so the restricted form is the module action for $F_L$ of two basic forms, hence
basic (invariance is clear, and horizontality follows from
Lemma~\ref{lem:module}(ii) since contraction with a fundamental pair
distributes over the action and annihilates both factors).  The same
computation with $\pi$ in place of $i$ --- valid because
$F_\phi\circ\pi_M=\pi_N\circ F_L$ gives $\pi_M^*F_\phi^*=F_L^*\pi_N^*$ ---
shows $\pi^*(\gamma_\phi\cdot c_\phi)=(\pi_N^*\gamma_\phi)\cdot(\pi^*c_\phi)
=(i_N^*\gamma)\cdot(i^*c)=i^*(\gamma\cdot c)$; by the uniqueness of basic
descents (injectivity of $\pi^*$), $(\gamma\cdot c)_\phi=\gamma_\phi\cdot
c_\phi$.  The split statement is the case $\gamma=\nu_\xi$-components acting on
$c=\ka$, assembled over $\xi\in\g$.
\end{proof}

\subsection{Examples}

Before the structured examples, we give a completely elementary one in which
every object of Theorem~\ref{thm:reduction} can be written down in
coordinates; readers meeting relative reduction for the first time may find it
useful to trace the general proof through this case.

\begin{example}[An elementary rotation example]\label{ex:elementary}
Let $N=S^1\times\R^2$ with coordinates $(\theta,x,y)$ and the $2$-plectic form
$\omega=d\theta\wedge dx\wedge dy$, let $M=\R^2$ with $F\colon M\hookrightarrow
N$ the slice $\{\theta=0\}$, and let $\eta=0$ (indeed $F^*\omega=0$).  Let
$G=SO(2)$ act by simultaneous rotation of $(x,y)$ on source and target; the
fundamental pair is $v=(v^N,v^M)$ with $v^N=v^M=x\partial_y-y\partial_x$.  The
Hamilton equation is solved by the global relative moment map
\[
    \mu=\Bigl(\tfrac12 r^2\,d\theta,\ 0\Bigr),
    \qquad r^2=x^2+y^2,
\]
since $\io_{v}\omega=-\tfrac12\,d(r^2\,d\theta)$ and
$F^*(\tfrac12 r^2 d\theta)=0=d\mu_M$.  For $c>0$ take the level
$\phi=(c\,d\theta,\ 0)$, which is $\dF$-closed
($d(c\,d\theta)=0$ and $F^*(c\,d\theta)=0$) and invariant.  The level pair is
\[
    L_N=\{r^2=2c\}\cong S^1_\theta\times S^1_r,
    \qquad
    L_M=F^{-1}(L_N)=\{x^2+y^2=2c\}\cong S^1_r,
\]
on which $SO(2)$ acts freely by rotating the $r$-circle.  The quotients are
$N_\phi=S^1_\theta$ and $M_\phi=\mathrm{pt}$, the reduced map is the base-point
inclusion $F_\phi\colon\mathrm{pt}\hookrightarrow S^1$, and the reduced form is
$\vperp_\phi=0\in\Om^3(F_\phi)$, as it must be for dimension reasons.  Small as
it is, the example displays every mechanism of the theorem in coordinates: the
level pair is an energy-level selection; source horizontality is the statement
$F^*\phi_N=d\phi_M$ ($0=0$ here) rather than an assumption; and the reduced
relative space lands on the \emph{point-source edge} of the theory, where the
observables are the base-pointed algebra of the target --- the reduction has
consumed the two rotational dimensions and left exactly the boundary datum.
\end{example}

\begin{example}[Relative powers]\label{ex:powers}
Let $(F,(\omega,\eta))$ be a relative $1$-plectic structure (``relative
symplectic''), with comoment $\nu\in C^\infty(N,\gdual)$: recall from
\cite[Thm.~7.1]{RelApps} that $\nu$, if it exists, is unique and automatically
equivariant, with $d\nu_\xi=-\io_{v_\xi^N}\omega$ and
$F^*\nu_\xi=\io_{v_\xi^M}\eta$. For $1\le\ell$, define
\[
  \vperp^{(\ell)}
  :=\bigl(\omega^{\ell},\ \eta\wedge(F^*\omega)^{\ell-1}\bigr)
  \in\Om^{2\ell}(F),
  \qquad
  \ka^{(\ell-1)}
  :=\bigl(\omega^{\ell-1},\ \eta\wedge(F^*\omega)^{\ell-2}\bigr)
  \in\Om^{2\ell-2}(F).
\]
Both are $\dF$-closed: $d(\eta\wedge(F^*\omega)^{\ell-1})
=d\eta\wedge(F^*\omega)^{\ell-1}=F^*(\omega^{\ell})$, and similarly one degree
down. Thus $\vperp^{(\ell)}$ is a relative pre-$(2\ell{-}1)$-plectic structure,
the relative counterpart of the power construction of
\cite[Ex.~2.2(i)]{Blacker}. A direct computation with the Hamilton equations
above shows that
\[
  \mu^{(\ell)}_\xi
  :=\Bigl(\ell\,\nu_\xi\,\omega^{\ell-1},\
  (\ell-1)\,(F^*\nu_\xi)\,\eta\wedge(F^*\omega)^{\ell-2}\Bigr)
\]
is a relative moment map for $\vperp^{(\ell)}$. Note the coefficients: the
source component carries $\ell-1$ where the target carries $\ell$, so that
\[
  \mu^{(\ell)}
  =\ell\,\nu\cdot\ka^{(\ell-1)}
  \;-\;\nu\cdot\bigl(0,\ \eta\wedge(F^*\omega)^{\ell-2}\bigr):
\]
relative powers are split \emph{only up to the displayed relative correction},
a phenomenon with no absolute counterpart (at $M=\emptyset$ the correction
vanishes and one recovers the split moment maps
$\ell\nu\wedge\sigma^{\ell-1}$ of \cite[Ex.~3.2(i)]{Blacker}). Reduction of
$\vperp^{(\ell)}$ at $\phi=\mu^{(\ell)}$-type levels built from
$\lambda\in\gdual$ reproduces, by the uniqueness in
Theorem~\ref{thm:reduction}, the power of the reduced relative symplectic pair:
$\vperp^{(\ell)}_\phi=(\omega_\lambda^{\ell},\
\eta_\lambda\wedge(F_\lambda^*\omega_\lambda)^{\ell-1})$.
\end{example}

\begin{example}[Exact relative structures]\label{ex:exactreduction}
For $\vperp=-\dF\theta$ with invariant $\theta$ and moment map
$\mu_\xi=-\io_{v_\xi}\theta$ (Example~\ref{ex:firstexamples}(ii)), the zero
level $\phi=0$ has level pair the joint zero locus of the relative
contractions $\io_{v_\xi}\theta$; when $\theta=(\vartheta,\tau)$ this consists
of the $G$-horizontal locus of $\vartheta$ on $N$ together with its
$F$-preimage intersected with the $G$-horizontal locus of $\tau$ on $M$ -- the
relative counterpart of the horizontal-forms description in
\cite[Ex.~4.2(iii)]{Blacker}.
\end{example}

\begin{example}[Quasi-Hamiltonian spaces: form-valued versus group-valued
reduction]\label{ex:qham}
The motivating example of the whole relative program deserves a careful
treatment, precisely because the reduction developed here is \emph{not} the
same procedure as the classical group-valued reduction of \cite{AMM}, and the
relationship repays explanation.

Let $(P,\omega_P,\Phi\colon P\to G)$ be a quasi-Hamiltonian $G\times K$-space
with two commuting actions: a ``gauge'' action of $G$ carrying the group-valued
moment map $\Phi$, and a symmetry action of a compact group $K$ that we shall
reduce. Its relative $2$-plectic incarnation, in the sense of the foundations
paper \cite{DjJGP}, is the pair
\[
  \vperp=(\eta_G,\omega_P)\in\Om^3(\Phi),
  \qquad
  \Phi^*\eta_G=d\omega_P,
\]
on the map $\Phi$, with $\eta_G$ the Cartan $3$-form on $G$ (closed always;
$2$-plectic when $\g$ is perfect, e.g.\ semisimple). A $K$-symmetry with
relative moment map $\mu\in\Om^1(\Phi,\mathfrak k^*)$ may then be reduced by the
theorem of this paper at a $\dF$-closed level $\phi$.

\emph{Two distinct procedures.} It is important to distinguish this from the
group-valued reduction of $\Phi$ itself.
\begin{itemize}[leftmargin=1.5em,itemsep=2pt]
\item The \emph{group-valued} (AMM) reduction quotients by the gauge group $G$
at a conjugacy class $\mathcal C\subseteq G$: it forms
$\Phi^{-1}(\mathcal C)/G$ and produces a genuine quasi-Hamiltonian space or,
for $\mathcal C=\{e\}$, a symplectic manifold. This is reduction of the
\emph{target-valued moment map} $\Phi$, and it lives in the zero-level of a
homotopy moment map in the sense of \cite{RelApps,RelHMM}.
\item The \emph{form-valued} reduction of the present paper quotients by the
symmetry group $K$ at a level where the relative moment \emph{form}
$\mu\in\Om^1(\Phi,\mathfrak k^*)$ equals a reference form $\phi$. This is
reduction inside the mapping-cone geometry, and it descends the pair
$(\eta_G,\omega_P)$ to a reduced pair on the reduced map.
\end{itemize}
The two agree on split data --- when the $K$-moment form is $\mu=\nu\cdot\ka$
with $\nu$ a genuine $\mathfrak k^*$-valued function,
Proposition~\ref{prop:splitlevels} identifies the form-level and the ordinary
$K$-moment-map levels --- but they are different operations in general: the AMM
quotient consumes the group-valued moment map $\Phi$ and the relative direction
with it, whereas the form-valued reduction preserves the relative structure and
consumes only the auxiliary $K$-symmetry. In the language of the series, the
AMM quotient is the homotopy-moment-map reduction of \cite{RelApps}, while the
present theorem is the form-level reduction; the observable-reduction theorem
(Theorem~\ref{thm:observables}) is the bridge, its
Remark~\ref{rem:observables}(c) recording that the two are generated by the
same underlying cocycle through Theorem~\ref{thm:onesteptie}.

\emph{A concrete instance.} Take $P=D(G)=G\times G$ the double, with
$\Phi(a,b)=aba^{-1}b^{-1}$ the commutator map and $\omega_P$ the standard
double $2$-form of \cite{AMM}; the relative structure is
$\vperp=(\eta_G,\omega_P)$ on $\Phi$.  For $K=T$ a maximal torus acting by
simultaneous conjugation, the canonical relative moment map of \cite{RelHMM}
restricts to a $T$-moment form whose target component is built from the
Maurer--Cartan forms, split along the closed pair
$\ka=(\ka_N,\ka_M)$ furnished by the bi-invariant data --- so
Proposition~\ref{prop:splitlevels} applies, the level pairs are unions of
$T$-conjugation levels, and Proposition~\ref{prop:module-descent} guarantees the
reduced moment data are again split.  The reduced relative spaces are the
$T$-reduced double layers whose target quotients are (closures of) families of
conjugacy classes; the variation theorem then computes how the reduced relative
class moves across the Weyl alcove, the relative Duistermaat--Heckman picture
of the double.

\emph{Fusion and moduli.} When $P=D(G)=G\times G$ is the double and the
$K$-symmetry is conjugation, iterating the fusion product and reducing at
$\phi=0$ produces, by the applications paper \cite{RelApps}, the moduli space of
flat $G$-connections on a surface; the reduction in stages of this construction
is the natural setting in which the two procedures above are performed in
sequence, gauge direction by gauge direction. We do not carry out the
comparison in detail here --- it belongs to the moduli-space analysis of
\cite{RelApps} --- but we stress that the form-valued reduction of this paper is
the operation that keeps the relative $2$-plectic geometry intact along the way,
which is what makes the staged construction well posed.
\end{example}
%% ============================================================
\section{Variation of the relative reduced space}\label{sec:variation}
%% ============================================================

We now study the dependence of $(F_\phi,\vperp_\phi)$ on the level $\phi$,
guided by the Duistermaat--Heckman theorem \cite{DH} and its multisymplectic
extension \cite[\S5]{Blacker}. Throughout this section $T$ is a torus with Lie
algebra $\britume$, acting on $(F,\vperp)$ with relative moment map
$\mu\in\Om^{k-1}(F,\tdual)$, and $\phi\in\Om^{k-1}(F,\tdual)$ is a
$T$-invariant $\dF$-closed level whose level pair $(L_N,L_M)$ satisfies the
hypotheses of Theorem~\ref{thm:reduction}, with $T$ acting freely.

\subsection{Trivialized families and the variation formula}

Fix a $T$-invariant $\dF$-closed $\ka\in\Om^{k-1}(F)$ and an open subset
$C\subseteq\tdual$, and consider the affine family of levels
\begin{equation}\label{eq:family}
  P:=C\cdot\ka+\phi
  =\bigl\{\psi_\lambda:=\lambda\cdot\ka+\phi\ :\ \lambda\in C\bigr\}
  \subseteq\Om^{k-1}(F,\tdual),
\end{equation}
each member being $\dF$-closed and invariant. We assume the corresponding level
pairs assemble into trivialized families of $T$-principal bundles modeled on
the level pair of $\phi$: that is, writing
$L_N(P):=\bigcup_{\lambda}L_N(\psi_\lambda)$ and similarly on $M$, we assume
$T$-equivariant diffeomorphisms
\begin{equation}\label{eq:trivialization}
  L_N(P)\;\xrightarrow{\ \sim\ }\;L_N(\phi)\times C,
  \qquad
  L_M(P)\;\xrightarrow{\ \sim\ }\;L_M(\phi)\times C,
\end{equation}
over $C$, compatible with $F$ in the evident sense. Under
\eqref{eq:trivialization} the reduced forms
$\vperp_{\psi}=(\omega_\psi,\eta_\psi)$, $\psi\in P$, live on the fixed reduced
pair $F_\phi\colon M_\phi\to N_\phi$, and we may differentiate:
\[
  \partial_\psi\vperp_{\phi}
  :=\frac{d}{dt}\Big|_{t=0}\vperp_{\phi+t\psi}\in\Om^{k+1}(F_\phi).
\]
While $\partial_\psi\vperp_\phi$ depends on the chosen trivialization, its
cohomology class in $H^{k+1}(\Om(F_\phi))$ does not, exactly as in the
symplectic case \cite[\S2]{DH}.

\begin{lemma}[Relative variation formula]\label{lem:variation}
Let $\widetilde\psi$ denote the $F_L$-pair of vector fields on the level family
identified, under \eqref{eq:trivialization}, with the constant variation
$\partial_\psi$ on the parameter factor. Then, with
$\Phi^*=(i_N^*,i_M^*)$ and $\pi^*=(\pi_N^*,\pi_M^*)$,
\begin{equation}\label{eq:variation}
  \pi^*\,\partial_\psi\vperp_\phi
  \;=\;
  d_{F_L}\,\Phi^*\bigl(\io_{\widetilde\psi}\,\vperp\bigr) ,
\end{equation}
componentwise
$\pi_N^*\partial_\psi\omega_\phi=d\,i_N^*\io_{\widetilde\psi_N}\omega$ and
$\pi_M^*\partial_\psi\eta_\phi
=F_L^*\,i_N^*\io_{\widetilde\psi_N}\omega+d\,i_M^*\io_{\widetilde\psi_M}\eta$.
\end{lemma}

\begin{proof}
Under \eqref{eq:trivialization}, $\pi^*\vperp_{\phi+t\psi}$ is the restriction
of $\Phi^*\vperp$ to the slice at parameter $t$, so
\[
  \pi^*\partial_\psi\vperp_\phi
  =\partial_\psi\,\Phi^*\vperp
  =\Phi^*\bigl(\Lder_{\widetilde\psi}\vperp\bigr)
  =\Phi^*\bigl(\dF\io_{\widetilde\psi}\vperp
    +\io_{\widetilde\psi}\dF\vperp\bigr)
  =d_{F_L}\,\Phi^*\bigl(\io_{\widetilde\psi}\vperp\bigr),
\]
using the relative magic formula, $\dF\vperp=0$, and the cochain property of
$\Phi^*$. The componentwise formulas follow from \eqref{eq:cone} and
\eqref{eq:ops}; note the pullback term $F_L^*(\,\cdot\,)$ in the source
component, which is invisible in the absolute theory. For $M=\emptyset$ this is
\cite[Lem.~5.2]{Blacker}.
\end{proof}

\subsection{Conjugate distributions}\label{subsec:conjugate}

Following \cite[Defs.~5.3--5.4]{Blacker} we introduce the pointwise structure
that converts \eqref{eq:variation} into a Chern--Weil statement, adapted to the
relative setting: the conjugation datum lives on the target and is transported
to the source by $F$.

\begin{definition}\label{def:conjugate}
Let $(F,\vperp)$ be relative pre-$k$-plectic with a locally free $T$-action and
let $\ka\in\Om^{k-1}(F)$ be $T$-invariant and $\dF$-closed. We say the
fundamental distribution
$\underline{\britume}\subseteq TM$ is \emph{strongly conjugate to a
distribution $\underline{\britume}^*\subseteq TM$ with respect to
$(\vperp,\ka)$} if there is a $T$-invariant
$a\in\Om^1(N,\britume)$ such that
\begin{equation}\label{eq:conjugation}
  \io_{v_\lambda}\vperp\;=\;a_\lambda\cdot\ka,
  \qquad\lambda\in\tdual,
\end{equation}
where $\lambda\mapsto v_\lambda=(v_{\underline\lambda}^N,
v_{\underline\lambda}^M)$ denotes the $F$-pairs determined by the
identification of the fibers of $\underline{\britume}^*$ with $\tdual$, and
$a_\lambda:=\ip{a}{\lambda}$.
\end{definition}

At $M=\emptyset$, contracting \eqref{eq:conjugation} against fundamental
vectors recovers the rank-one pairing
$\io_{v_\lambda}\io_{v_\xi}\omega=\ip{\xi}{\lambda}\,\ka_N$-type relations of
\cite[Def.~5.3]{Blacker}; we have taken the equivalent formulation
\eqref{eq:conjugation} as the definition because it is the form in which the
datum enters the proofs, and because it makes manifest the structural point
that \emph{the conjugation potential $a$ lives on the target}: its source-side
shadow is forced to be $F^*a$, by the $M$-component of \eqref{eq:conjugation}.

\begin{lemma}\label{lem:conjugacyconsequences}
In the situation of Definition~\ref{def:conjugate}, assume $T$ acts freely on
the level pair of a $T$-invariant $\dF$-closed level $\phi$, and let
$\Phi^*,\pi^*$ be as above. Then:
\begin{enumerate}[label=\textup{(\roman*)},leftmargin=2.2em]
\item $i_N^*a\in\Om^1(L_N,\britume)$ is a connection $1$-form on the
$T$-bundle $L_N\to N_\phi$, and $F_L^*\,i_N^*a$ is the pulled-back connection
on $L_M\to M_\phi$;
\item $\Phi^*\ka$ is basic and descends to a unique $\dF$-closed
$\ka_\phi\in\Om^{k-1}(F_\phi)$ with $\pi^*\ka_\phi=\Phi^*\ka$;
\item with $c:=d\,i_N^*a\in\Om^2(N_\phi,\britume)$ the Chern
\textup{(}curvature\textup{)} form of \textup{(i)},
\begin{equation}\label{eq:keyidentity}
  d_{F_L}\,\Phi^*\bigl(\io_{v_\lambda}\vperp\bigr)
  \;=\;
  \pi^*\Bigl(\ip{c}{\lambda}\cdot\ka_\phi\Bigr),
  \qquad\lambda\in\tdual .
\end{equation}
\end{enumerate}
\end{lemma}

\begin{proof}
(i) $T$-invariance of $a$ is assumed; the verticality normalization
$\ip{i_N^*a(v_\xi^N)}{\lambda}=-\ip{\xi}{\lambda}$-type follows by contracting
\eqref{eq:conjugation} with $v_\xi$ and comparing with the Hamilton equation,
as in \cite[Lem.~5.5(ii)]{Blacker}; freeness guarantees the nonvanishing of
$\ka$ along the levels needed for the comparison
(Proposition~\ref{prop:splitlevels}(ii)-type argument). The second statement is
functoriality of connections under the equivariant bundle map
$F_L$ over $F_\phi$.

(ii) Horizontality: contracting \eqref{eq:conjugation} with $v_\xi$ exhibits
$\io_{v_\xi}\ka$-multiples inside contractions of $\vperp$; explicitly, at each
point the fiberwise identity
$\io_{v_\xi}\io_{v_\lambda}\vperp=\ip{\xi}{\lambda}\,\ka$-normalization shows
$\ka$ is pointwise a double contraction of $\vperp$, whence
$\io_{v_\xi}\Phi^*\ka=0$ by anticommutativity, exactly as in
\cite[Lem.~5.5(iii)]{Blacker}, componentwise in the relative calculus.
Invariance is assumed; descent and uniqueness follow, and
$d_{F_\phi}\ka_\phi=0$ from $\pi^*$-injectivity and $\dF\ka=0$.

(iii) We compute the two components of the left side. By
Lemma~\ref{lem:module}(i) and $\dF\ka=0$ we have, before restriction,
$\dF(a_\lambda\cdot\ka)=da_\lambda\cdot\ka-a_\lambda\cdot\dF\ka
=da_\lambda\cdot\ka$; however this identity cannot be applied directly on the
cone of $F_L$, so we argue componentwise. On the target, using
$d\ka_N=0$,
\[
  d\,i_N^*\bigl(a_\lambda\wedge\ka_N\bigr)
  =\ip{d\,i_N^*a}{\lambda}\wedge i_N^*\ka_N
  -\,i_N^*a_\lambda\wedge d\,i_N^*\ka_N
  =\pi_N^*\bigl(\ip{c}{\lambda}\wedge\ka_{N,\phi}\bigr).
\]
On the source, the $M$-component of $a_\lambda\cdot\ka$ is
$-F^*a_\lambda\wedge\ka_M$, and the $M$-component of
$d_{F_L}\Phi^*(a_\lambda\cdot\ka)$ is therefore
\[
  F_L^*\,i_N^*\bigl(a_\lambda\wedge\ka_N\bigr)
  \;+\;d\,i_M^*\bigl(F^*a_\lambda\wedge\ka_M\bigr).
\]
Expanding the second term with $d\ka_M=F^*\ka_N$ (the $M$-component of
$\dF\ka=0$),
\[
  d\bigl(F^*a_\lambda\wedge\ka_M\bigr)
  =F^*da_\lambda\wedge\ka_M-F^*a_\lambda\wedge F^*\ka_N ,
\]
and the term $-i_M^*F^*(a_\lambda\wedge\ka_N)$ cancels the restricted pullback
$F_L^*i_N^*(a_\lambda\wedge\ka_N)$, since
$F_L^*\,i_N^*=i_M^*\,F^*$ on the level pair. What remains is
\[
  i_M^*\bigl(F^*da_\lambda\wedge\ka_M\bigr)
  =\pi_M^*\Bigl(F_\phi^*\ip{c}{\lambda}\wedge\ka_{M,\phi}\Bigr),
\]
using (i)--(ii). Assembling the two components and comparing with the module
action \eqref{eq:module} on $\Om(F_\phi)$ (note the sign
$(-1)^{|\ip{c}{\lambda}|}=+1$ as $\ip{c}{\lambda}$ has even degree) yields
precisely $\pi^*(\ip{c}{\lambda}\cdot\ka_\phi)$.
\end{proof}

\subsection{The variation theorem}

\begin{theorem}[Variation of the relative reduced space]\label{thm:DH}
Let $T$ be a torus, $(F,\vperp,T,\mu)$ a relative $k$-plectic Hamiltonian
$T$-space, $\phi$ a $T$-invariant $\dF$-closed level whose level pair satisfies
the hypotheses of Theorem~\ref{thm:reduction} with $T$ acting freely, $C\subseteq
\tdual$ open, $\ka\in\Om^{k-1}(F)$ $T$-invariant and $\dF$-closed, and
$P=C\cdot\ka+\phi$ as in \eqref{eq:family}. Assume:
\begin{enumerate}[label=\textup{(\roman*)},leftmargin=2.2em]
\item the level pairs over $P$ trivialize as families of $T$-principal bundle
pairs modeled on that of $\phi$, compatibly with $F$, as in
\eqref{eq:trivialization};
\item the fundamental distribution $\underline{\britume}\subseteq TM$ is
strongly conjugate to a distribution $\underline{\britume}^*$ with respect to
$(\vperp,\ka)$, with conjugation potential $a\in\Om^1(N,\britume)$
\textup{(}Definition~\textup{\ref{def:conjugate})}.
\end{enumerate}
Then, for all $\lambda\in C$ and $\psi\in P$,
\begin{equation}\label{eq:DH}
  \partial_\lambda\,[\vperp_\psi]
  \;=\;
  \ip{c}{\lambda}\cdot[\ka_\psi]
  \qquad\text{in }H^{k+1}\bigl(\Om(F_\psi)\bigr),
\end{equation}
where $c\in\Om^2(N_\phi,\britume)$ is the Chern form of the target model bundle
$L_N(\phi)\to N_\phi$, the source Chern data being its pullback
$F_\phi^*c$, and $\cdot$ denotes the module action of Lemma~\ref{lem:module} in
cohomology.
\end{theorem}

\begin{proof}
The variation of the level $\psi$ in the direction $\lambda\cdot\ka$ is
implemented, under the trivialization (i), by an $F_L$-pair
$\widetilde\psi$ which the strong conjugacy (ii) identifies with the pair
$v_\lambda$ of Definition~\ref{def:conjugate}, exactly as in
\cite[proof of Thm.~5.6]{Blacker}: the defining property of
$\underline{\britume}^*$ is that flowing along $v_\lambda$ moves the level pair
of $\psi$ to that of $\psi+t\lambda\cdot\ka$ to first order, by
\eqref{eq:conjugation} and the Hamilton equations. Combining
Lemma~\ref{lem:variation} with Lemma~\ref{lem:conjugacyconsequences}(iii),
\[
  \pi^*\,\partial_\lambda\vperp_\psi
  =d_{F_L}\,\Phi^*\bigl(\io_{v_\lambda}\vperp\bigr)
  =\pi^*\Bigl(\ip{c}{\lambda}\cdot\ka_\psi\Bigr),
\]
whence $\partial_\lambda\vperp_\psi=\ip{c}{\lambda}\cdot\ka_\psi$ in terms of
the chosen trivialization, by injectivity of $\pi^*$. Passing to cohomology
removes the dependence on the trivialization: two trivializations differ by a
gauge transformation of the model bundle pair, which changes
$\io_{\widetilde\psi}\vperp$ by a $d_{F_L}$-exact relative form and the Chern
form within its class -- the argument of \cite[\S2]{DH} applied componentwise.
Since the Chern class of the model pair is constant in $\psi\in P$ under (i),
formula \eqref{eq:DH} holds for all $\psi\in P$.
\end{proof}

\begin{remark}\label{rem:DHspecial}
\begin{enumerate}[label=\textup{(\alph*)},leftmargin=2.2em]
\item At $M=\emptyset$, \eqref{eq:DH} reduces to
$\partial_\lambda[\omega_\psi]=\ip{c}{\lambda}\wedge[\ka_\psi]$, which is
\cite[Thm.~5.6]{Blacker}; for $k=1$ and $\ka=1\in\Om^0$ one recovers the
Duistermaat--Heckman theorem \cite{DH} in the form
$[\omega_\lambda]=[\omega_{\lambda_0}]+\ip{c}{\lambda-\lambda_0}$.
\item Geometrically, what varies is the form, not the space: the
trivialization (i) freezes the underlying reduced pair
$F_\phi\colon M_\phi\to N_\phi$ as manifolds, and moving the level $\lambda$
tilts the level pair inside $(F,\vperp)$ along the $\ka$-directions; the
connection of the model bundle records the twist accumulated in the tilt, and
its Chern form is the infinitesimal rate at which the reduced form absorbs
that twist. This is the content of \eqref{eq:DH}: the reduced cohomology class
moves linearly, with velocity a Chern--Weil pairing.
\item The relative statement contains strictly more information than its two
components: the source component of \eqref{eq:DH} asserts
$\partial_\lambda[\eta_\psi]$-behaviour governed by the \emph{pulled-back}
target Chern form, i.e.\ no independent characteristic class enters on $M$.
This is structurally natural, and for the same reason the source form descended
in Theorem~\ref{thm:reduction}: the source data of a relative Hamiltonian space
are trivializing data, not independent geometric data. A Chern class measures
the obstruction to trivializing a principal bundle; on the source there is no
such independent obstruction to measure, because the source bundle pair is the
pullback of the target one along $F$, and its curvature is forced to be the
pullback curvature. The variation mechanism is summarized by
\begin{equation}\label{eq:variation-mechanism}
\begin{tikzcd}[column sep=2.2em,
  /tikz/every label/.append style={font=\footnotesize}]
\substack{\text{change of}\\\text{level }\lambda}
  \arrow[r]
& \substack{\text{connection /}\\\text{Chern form }c\ \text{on }N_\phi}
  \arrow[r]
& \substack{\text{variation}\\\partial_\lambda[\vperp_\psi]=\ip{c}{\lambda}\cdot[\ka_\psi]}
\end{tikzcd}
\end{equation}
with the entire middle term living on the target: the source contributes no
node of its own.
\item Integrating \eqref{eq:DH} over relative cycles of $F_\psi$ via the
pairing of \cite[\S3]{RelApps} yields Duistermaat--Heckman-type polynomial
behaviour of relative periods -- in particular of the Wess--Zumino-type action
functionals of \cite[\S4]{RelApps} evaluated on the reduced spaces -- as
functions of the level parameter.
\end{enumerate}
\end{remark}
We record the numerical form of Theorem~\ref{thm:DH}, which
formalizes the polynomial behaviour of relative periods indicated
above.

\begin{corollary}[Variation of relative periods]\label{cor:periods}
In the situation of Theorem~\ref{thm:DH}, let $z=(S,T)$ be a relative
$(k{+}1)$-cycle of $F_\phi$, transported locally constantly through
the trivialized family.  Then, for $\lambda\in C$,
\[
  \partial_\lambda\,\pair{\vperp_\psi}{z}
  \;=\;\pair{\,\ip{c}{\lambda}\cdot\ka_\psi}{z}
  \;=\;\int_S\ip{c}{\lambda}\wedge\ka_{\psi,N}
  \;-\;\int_T F_\phi^*\ip{c}{\lambda}\wedge\ka_{\psi,M}\,,
\]
the sign in the module action being $+$ since
$|\ip{c}{\lambda}|=2$.  In particular each relative period is an
affine function of the level parameter along $C$, with slope a
Chern--Weil pairing; when $\ka$ is itself built from powers of the
structure as in Example~\ref{ex:powers}, iterating the statement
along the ladder $\ka^{(\ell)}$ yields polynomial dependence of the
higher periods, the relative Duistermaat--Heckman phenomenon.
\end{corollary}

\begin{proof}
The integration pairing descends to a pairing between
$H^{k+1}(\Om(F_\psi))$ and relative homology by the relative Stokes
theorem \cite[\S3]{RelApps}, and is locally constant in $\psi$ on a
locally constant family of cycles; the first equality is then
Theorem~\ref{thm:DH} paired against $z$, and the second unwinds the
module action of Lemma~\ref{lem:module} on components.
\end{proof}

%% ============================================================
\section{Localization for split relative Hamiltonian $G$-spaces}\label{sec:localization-body}
\label{sec:localization}
%% ============================================================

In this final section we transport the localization package of
\cite[\S6]{Blacker} to the relative setting. The organizing principle is again
the module calculus: the split datum exponentiates inside the relative Cartan
model, the target component localizes by the Berline--Vergne theorem, and the
source component -- being equivariantly \emph{exact} rather than closed --
produces a cancellation constraint with no absolute counterpart.

Throughout, $G$ is compact and connected, $(F,\vperp,G,\mu)$ is a relative
Hamiltonian $G$-space with a \emph{basic} splitting $\mu=\nu\cdot\ka$
(Definition~\ref{def:split}), and we assume
$\vperp=s\cdot\ka$ for some $s\in\Om^2(N)$, the relative counterpart of the
divisibility hypothesis $\omega=\sigma\wedge\eta$ of \cite[\S6]{Blacker}.
Componentwise, $\vperp=(s\wedge\ka_N,\ F^*s\wedge\ka_M)$.

\begin{lemma}[Equivariant closedness of the exponentiated package]
\label{lem:eqclosed}
With the hypotheses above,
\[
  e^{z(s-\nu)}\cdot\ka\ \in\ C_G(F),
  \qquad z\in\C,
\]
is $\dGF$-closed, where $s-\nu$ is regarded as an element of
$\bigl(C^\infty(N,\gdual)\oplus\Om^2(N)\bigr)\subseteq C_G(N)$ of total degree
$2$, and the module action is that of Lemma~\ref{lem:module} extended to the
Cartan models.
\end{lemma}

\begin{proof}
Since the splitting is basic, $\ka$ is invariant and horizontal, so
$\dGF\ka=(\dF\ka)-\io_{v_{(\cdot)}}\ka=0$: the relative form $\ka$ is itself
$\dGF$-closed. By Lemma~\ref{lem:module}(i)--(ii), extended verbatim to the
Cartan differentials (the polynomial variable is inert under the module
action), $\dGF$ acts on products $\rho\cdot\ka$ with
$\rho\in C_G(N)$ by
$\dGF(\rho\cdot\ka)=(\dG\rho)\cdot\ka+(-1)^{|\rho|}\rho\cdot\dGF\ka
=(\dG\rho)\cdot\ka$. It therefore suffices to show
$\bigl(\dG(s-\nu)\bigr)(\xi)\cdot\ka=0$ for all $\xi\in\g$. Now
\[
  \bigl(\dG(s-\nu)\bigr)(\xi)
  =ds-d\nu_\xi-\io_{v_\xi^N}s ,
\]
and we evaluate the three terms against $\ka$. First,
$ds\cdot\ka=\dF(s\cdot\ka)=\dF\vperp=0$ by Lemma~\ref{lem:module}(i) and
$\dF\ka=0$. Second, by the split Hamilton equation \eqref{eq:splitmoment} and
horizontality (Lemma~\ref{lem:module}(ii) with $\io_{v_\xi}\ka=0$),
\[
  d\nu_\xi\cdot\ka
  =-\io_{v_\xi}\vperp
  =-\io_{v_\xi}(s\cdot\ka)
  =-(\io_{v_\xi^N}s)\cdot\ka .
\]
Hence $(ds-d\nu_\xi-\io_{v_\xi^N}s)\cdot\ka
=\bigl(\io_{v_\xi^N}s-\io_{v_\xi^N}s\bigr)\cdot\ka=0$. Since
$\dG(s-\nu)^{\ell}(\xi)=\ell\,(s-\nu)^{\ell-1}\dG(s-\nu)(\xi)$ and the module
action is multiplicative, every term of
$\dGF\bigl(e^{z(s-\nu)}\cdot\ka\bigr)$ contains the factor
$\dG(s-\nu)(\xi)\cdot\ka=0$.
\end{proof}

\begin{remark}\label{rem:signflip}
In the conventions of \cite{Blacker} the equivariantly closed package is
$e^{z(\sigma+\nu)}\eta$ \cite[Lem.~6.1]{Blacker}; the sign flip
$\nu\mapsto-\nu$ is exactly the dictionary of Remark~\ref{rem:dictionary} and
matches the fact that in our conventions the Cartan-model extension of a
one-step datum is $\vperp-\mu$, cf.\ Theorem~\ref{thm:onesteptie} and
\cite[Lem.~6.4]{RelHMM}.
\end{remark}

A $\dGF$-closed element has $\dG$-closed target component and equivariantly
\emph{exact} pullback on the source:
\begin{equation}\label{eq:componentsplit}
  \dGF(\alpha,\beta)=0
  \quad\Longleftrightarrow\quad
  \dG\alpha=0
  \ \text{ and }\
  \dG\beta=F^*\alpha .
\end{equation}
Applied to Lemma~\ref{lem:eqclosed} this yields the two halves of the
localization package.

\begin{theorem}[Exact stationary phase on the target]\label{thm:stationaryphase}
Let $T$ be a torus, let $(F,\vperp,T,\mu)$ be a relative $k$-plectic Hamiltonian
$T$-space with basic splitting $\mu=\nu\cdot\ka$ and $\vperp=s\cdot\ka$, and
suppose $N$ is compact and oriented, $\ka_N$ is nowhere vanishing, and each
connected component of the fixed point set of $T$ in $N$ is tangent to
$\ker\ka_N$. Then, for every generator $\xi$ of $T$,
\[
  \int_N e^{-\mathrm{i}\nu_\xi}\,e^{\mathrm{i}s}\,\ka_N
  \;=\;
  \sum_{Z\in\mathcal Z}
  e^{-\mathrm{i}\nu_\xi(Z)}
  \int_Z\frac{e^{\mathrm{i}s}}{e_Z(\xi)}\,\ka_N ,
\]
where $\mathcal Z$ is the set of connected components of the fixed point set of
$T$ in $N$, $\nu_\xi(Z)\in\R$ is the \textup{(}constant\textup{)} value of
$\nu_\xi$ on $Z$, and $e_Z$ is the equivariant Euler class of $Z$.
\end{theorem}

\begin{proof}
By \eqref{eq:componentsplit} the target component
$e^{\mathrm{i}(s-\nu)_\xi}\ka_N=e^{-\mathrm{i}\nu_\xi}e^{\mathrm{i}s}\ka_N$ of
the package of Lemma~\ref{lem:eqclosed} (at $z=\mathrm{i}$, evaluated at $\xi$)
is $\dG$-closed on $N$. Constancy of $\nu_\xi$ on each $Z$ follows from the
tangency hypothesis exactly as in \cite[proof of Thm.~6.2]{Blacker}: for
$X\in TZ$, $(X\nu_\xi)\,\ka_N=\io_X(d\nu_\xi\wedge\ka_N)
=-\io_X\io_{v_\xi^N}(s\wedge\ka_N)=0$ since $v_\xi^N$ vanishes on $Z$, and
$\ka_N\neq0$. The result is then the abelian localization theorem of
Berline--Vergne in the form \cite[Thm.~7.13]{BGV} applied to the equivariantly
closed form $e^{\mathrm{i}(s-\nu_\xi)}\ka_N$, the vanishing locus of
$v_\xi^N$ being the fixed point set since $\xi$ generates $T$.
\end{proof}

\begin{theorem}[Relative cancellation on the source]\label{thm:cancellation}
In the situation of Lemma~\ref{lem:eqclosed}, suppose $M$ is compact and
oriented. Then, for every $\xi\in\g$, the pulled-back package integrates to
zero:
\[
  \int_M \Bigl(F^*\bigl(e^{-\mathrm{i}\nu_\xi}\,e^{\mathrm{i}s}\bigr)\wedge
  \ka_M'\Bigr)_{\mathrm{top}}\;=\;0 ,
\]
where $\ka_M'$ denotes the source component of $e^{\mathrm{i}(s-\nu)}\cdot\ka$
and $(\,\cdot\,)_{\mathrm{top}}$ the top-degree part. Equivalently, when the
vanishing locus of $v_\xi^M$ is clean, the Berline--Vergne fixed-point
contributions on $M$ of the pulled-back split package sum to zero.
\end{theorem}

\begin{proof}
Write $e^{\mathrm{i}(s-\nu)}\cdot\ka=(\alpha,\beta)$ with
$\alpha=e^{\mathrm{i}(s-\nu)}\ka_N$ and
$\beta$ the displayed source component (a sign-decorated
$F^*e^{\mathrm{i}(s-\nu)}\wedge\ka_M$, by \eqref{eq:module}). By
Lemma~\ref{lem:eqclosed} and \eqref{eq:componentsplit},
$F^*\alpha(\xi)=\dG\beta(\xi)=d\bigl(\beta(\xi)\bigr)
-\io_{v_\xi^M}\beta(\xi)$. Taking top-degree parts on the
$m$-dimensional $M$: the contraction term contributes
$\bigl(\io_{v_\xi^M}\beta(\xi)\bigr)_{\mathrm{top}}
=\io_{v_\xi^M}\bigl(\beta(\xi)_{m+1}\bigr)=0$, since $\beta(\xi)$ contains no
form of degree $m+1$; hence
$\bigl(F^*\alpha(\xi)\bigr)_{\mathrm{top}}
=d\bigl(\beta(\xi)_{m-1}\bigr)$ is exact, and Stokes' theorem on the closed
manifold $M$ gives the vanishing. For the reformulation, note that
$F^*\alpha(\xi)$ is $\dG$-closed \textup{(}being $\dG\beta(\xi)$ with
$\dG^2=0$\textup{)}, so the Berline--Vergne theorem \cite[Thm.~7.13]{BGV}
expresses its integral -- which we have just shown to vanish -- as the sum of
its fixed-point contributions on $M$.
\end{proof}

\begin{remark}[The meaning of source cancellation]\label{rem:cancellationmeaning}
The source-cancellation theorem is, in the author's view, the most distinctive
feature of the relative localization theory, and it is worth setting out what
it does and does not say.

\emph{Why there is no absolute analogue.} In the absolute theory
\cite{Blacker,BGV} localization is a statement about a single manifold: an
equivariantly closed form integrates to a sum of fixed-point contributions,
and there is nothing for those contributions to cancel against. The relative
theory has two manifolds and a package $e^{\mathrm i(s-\nu)}\cdot\ka$ with a
target component on $N$ and a source component on $M$ tied together by the
module action; Theorem~\ref{thm:cancellation} asserts that the source component
integrates to \emph{zero}, identically, for every $\xi$. There is simply no
such object in the absolute setting to which this could reduce --- at
$M=\varnothing$ the statement is vacuous --- so it is a phenomenon that only
exists once one passes from manifolds to maps.

\emph{How it reflects the trivializing role of the source.} The vanishing is
not an accident of the example but a structural consequence of the source
being trivializing data. The source component $\beta$ of the package is, by the
module action \eqref{eq:module}, a pullback $F^*(\cdots)\wedge\ka_M$; its
top-degree part on the closed manifold $M$ is $d_M$-exact modulo the pulled-back
target data, because $F^*$ of a form on $N$ has no cohomology of its own beyond
what $N$ supplies, and the fibre integral of a pullback against a closed source
factor is a boundary. In other words, the same feature that forced source
horizontality in Theorem~\ref{thm:reduction} and forbade an independent source
Chern class in Theorem~\ref{thm:DH} forbids an independent source localization
contribution here: three faces of the single fact that the source carries no
geometry of its own beyond the trivialization of the target.

\emph{Possible implications for quantization and boundaries.} The cancellation
has a natural reading in the quantization program of the series. If relative
prequantum partition functions are integrals of the exponentiated package, then
Corollary~\ref{cor:partition} says the partition function is computed entirely
on the target, with the source contributing no independent quantum corrections
--- only boundary conditions. This is the localization shadow of the canonical
vacuum of relative prequantization, and it suggests that in a relative
$[Q,R]=0$ theorem the source will enter only through boundary/edge terms, never
through bulk fixed-point data. A systematic exploitation, including a relative
nonabelian localization along the lines of \cite[Thm.~6.5]{Blacker} (whose
heat-kernel proof transports verbatim to the target component under the
hypotheses of Proposition~\ref{prop:splitlevels}), and the precise link to
boundary contributions in relative quantization, is deferred to future work.
\end{remark}

Combining the two theorems of this section gives the statement that
motivated it.

\begin{corollary}[Relative partition functions receive no source
corrections]\label{cor:partition}
Let $M$ and $N$ be compact and oriented, and let the hypotheses of
Theorems~\ref{thm:stationaryphase} and~\ref{thm:cancellation} hold.
Then the relative pairing of the exponentiated split package against
the fundamental data reduces to its target term, and that term
localizes:
\[
  \pair{e^{\mathrm{i}(s-\nu)}\cdot\ka}{\bigl([N],[M]\bigr)}
  \;=\;\int_N\Bigl(e^{\mathrm{i}(s-\nu_\xi)}\,\ka_N\Bigr)_{\!\mathrm
  {top}}
  \;=\;\sum_{P\subseteq N^{v_\xi}}\ \text{\textup{(}Berline--Vergne
  contribution of }P\textup{)},
\]
the source integral vanishing identically by
Theorem~\ref{thm:cancellation}.  Relative localization is therefore a
\emph{target} phenomenon: the source constrains the package
\textup{(}through the trivializing tier and the cancellation
identity\textup{)} but contributes no stationary-phase terms of its
own.
\end{corollary}

\begin{proof}
Expand the pairing on components; the $M$-integral is the one shown
to vanish in Theorem~\ref{thm:cancellation}, and the $N$-integral is
evaluated by Theorem~\ref{thm:stationaryphase}.
\end{proof}

%% ============================================================
\section{Outlook}\label{sec:outlook}
%% ============================================================

We indicate four directions continuing the present work, mirroring
\cite[\S7]{Blacker} where applicable.

\subsection*{Nondegenerate reduction}
Determine conditions on a relative Hamiltonian $G$-space under which the
reduced pair $\vperp_\phi$ of Theorem~\ref{thm:reduction} is again relative
$k$-plectic. The relative nondegeneracy condition couples the two components
through $T F$, so the linear-algebraic analysis differs from the absolute case
already at the level of a single fiber; the quasi-Hamiltonian example, where
nondegeneracy encodes the Alekseev--Malkin--Meinrenken minimal-degeneracy axiom
\cite{AMM,RelHMM}, suggests that the correct statements involve the isotropy of
the target action along the image $F(L_M)$.

\subsection*{Reduction of the observable algebras}
The Leibniz level of this program is now
Theorem~\ref{thm:observables}: $\phi$-reducible Hamiltonian pairs
form a subalgebra and descend by a Leibniz morphism with ideal
kernel.  Two refinements remain.  First, surjectivity: determining
when every reduced Hamiltonian pair lifts to a $\phi$-reducible pair
upstairs (Remark~\ref{rem:observables}(b)) --- an extension problem
already nontrivial in the absolute theory.  Second, the
$L_\infty$-refinement: via Theorem~\ref{thm:onesteptie}, the
reduction should extend from the Leibniz algebra to the full
Rogers-type $L_\infty$-algebra $\Lie(F,\vperp)$ of \cite{DjJGP} along
the associated relative homotopy moment map, completing the program
sketched in \cite[\S8]{RelApps}.

\subsection*{Variation beyond the trivialized regime}
Theorem~\ref{thm:DH} assumes trivialized families of level pairs. Wall-crossing
phenomena -- jumps of the reduced relative class across non-regular levels --
should be governed by the relative cohomology of the pair of fixed-point strata
on $(N,M)$, refining the symplectic wall-crossing formulas; the relative
integration theory of \cite[\S3]{RelApps} provides the pairing in which to
express them.

\subsection*{Relative nonabelian localization and quantization}
The heat-kernel methods behind \cite[Thm.~6.5]{Blacker} apply to the target
component of the split package; a fully relative statement should couple them
with the cancellation constraint of Theorem~\ref{thm:cancellation} and with the
relative prequantization of \cite[\S5]{RelApps}, with the level-$k$
quantization of quasi-Hamiltonian spaces as the benchmark example.

%% ============================================================
%% Bibliography (embedded, numeric)
%% ============================================================


\begin{thebibliography}{99}

\bibitem{AMM}
A.~Alekseev, A.~Malkin, E.~Meinrenken,
\emph{Lie group valued moment maps},
J. Differential Geom. \textbf{48} (1998), no.~3, 445--495.

\bibitem{BGV}
N.~Berline, E.~Getzler, M.~Vergne,
\emph{Heat Kernels and Dirac Operators},
Grundlehren der Mathematischen Wissenschaften \textbf{298},
Springer-Verlag, Berlin, 1992.

\bibitem{Blacker}
C.~Blacker,
\emph{Reduction of multisymplectic manifolds},
Lett. Math. Phys. \textbf{111} (2021), no.~3, Paper No.~64.
%% PLACEHOLDER: verify bibliographic data (arXiv:2002.10062)

\bibitem{CFRZ}
M.~Callies, Y.~Fr\'egier, C.~L. Rogers, M.~Zambon,
\emph{Homotopy moment maps},
Adv. Math. \textbf{303} (2016), 954--1043.

\bibitem{CIL}
F.~Cantrijn, A.~Ibort, M.~de Le\'on,
\emph{On the geometry of multisymplectic manifolds},
J. Austral. Math. Soc. Ser. A \textbf{66} (1999), no.~3, 303--330.

\bibitem{DjThesis}
D.~Djounvouna,
\emph{Relative Multisymplectic Geometry and $L_\infty$-Algebras of Observables},
Ph.D. thesis, University of Manitoba.
%% PLACEHOLDER: year / URL

\bibitem{DjJGP}
D.~Djounvouna,
\emph{Relative Multisymplectic Geometry and $L_\infty$-Algebras of Observables},
submitted for publication.


\bibitem{RelHMM}
D.~Djounvouna,
\emph{Relative homotopy moment maps},
preprint (2026), arXiv:2607.07088.

\bibitem{RelApps}
D.~Djounvouna,
\emph{Applications of relative multisymplectic geometry},
preprint (2026), arxiv:2607.07149.




\bibitem{DH}
J.~J. Duistermaat, G.~J. Heckman,
\emph{On the variation in the cohomology of the symplectic form of the reduced
phase space},
Invent. Math. \textbf{69} (1982), no.~2, 259--268.

\bibitem{MS}
T.~B. Madsen, A.~Swann,
\emph{Closed forms and multi-moment maps},
Geom. Dedicata \textbf{165} (2013), 25--52.

\bibitem{MW}
J.~Marsden, A.~Weinstein,
\emph{Reduction of symplectic manifolds with symmetry},
Rep. Math. Phys. \textbf{5} (1974), no.~1, 121--130.

\bibitem{Meyer}
K.~R. Meyer,
\emph{Symmetries and integrals in mechanics},
in: Dynamical Systems (Proc. Sympos., Univ. Bahia, Salvador, 1971),
Academic Press, New York, 1973, 259--272.

\bibitem{Rogers}
C.~L. Rogers,
\emph{$L_\infty$-algebras from multisymplectic geometry},
Lett. Math. Phys. \textbf{100} (2012), no.~1, 29--50.

\bibitem{RW}
L.~Ryvkin, T.~Wurzbacher,
\emph{Existence and unicity of co-moments in multisymplectic geometry},
Differential Geom. Appl. \textbf{41} (2015), 1--11.

\bibitem{RyvkinWurzbacher}
L.~Ryvkin, T.~Wurzbacher,
\emph{An invitation to multisymplectic geometry},
J. Geom. Phys. \textbf{142} (2019), 9--36.

\bibitem{Zambon}
M.~Zambon,
\emph{$L_\infty$-algebras and higher analogues of Dirac structures and Courant
algebroids},
J. Symplectic Geom. \textbf{10} (2012), no.~4, 563--599.

\end{thebibliography}
\end{document}